\documentclass[11pt,draft,reqno]{amsart}
\input epsf.tex

\usepackage[all]{xy}
\usepackage{amsthm,array,amssymb,amscd,amsfonts,latexsym, url}

\renewcommand{\baselinestretch}{1.1}
\DeclareMathOperator{\catq}{\!/\!\!/\!}
\DeclareFontFamily{OT1}{wncyr}{\hyphenchar\font45}
\DeclareFontShape{OT1}{wncyr}{m}{n}{%
   <5> <6> <7> <8> <9> gen * wncyr
   <10> <10.95> <12> <14.4> <17.28> <20.74>  <24.88>wncyr10}{}
\DeclareFontShape{OT1}{wncyr}{m}{it}{%
   <5> <6> <7> <8> <9> gen * wncyi
   <10> <10.95> <12> <14.4> <17.28> <20.74> <24.88> wncyi10}{}
\DeclareFontShape{OT1}{wncyr}{m}{sc}{%
   <5> <6> <7> <8> <9> <10> <10.95> <12> <14.4>
   <17.28> <20.74> <24.88>wncysc10}{}
\DeclareFontShape{OT1}{wncyr}{b}{n}{%
   <5> <6> <7> <8> <9> gen * wncyb
   <10> <10.95> <12> <14.4> <17.28> <20.74> <24.88>wncyb10}{}
\input cyracc.def
\def\rus{\usefont{OT1}{wncyr}{m}{n}\cyracc\fontsize{9}{11pt}\selectfont}
\def\rusit{\usefont{OT1}{wncyr}{m}{it}\cyracc\fontsize{9}{11pt}\selectfont}
\def\russc{\usefont{OT1}{wncyr}{m}{sc}\cyracc\fontsize{9}{11pt}\selectfont}

\DeclareMathSizes{9}{9}{7}{5} % This only is different.

{\catcode`\@=11
\gdef\n@te#1#2{\leavevmode\vadjust{%
 {\setbox\z@\hbox to\z@{\strut#1}%
  \setbox\z@\hbox{\raise\dp\strutbox\box\z@}\ht\z@=\z@\dp\z@=\z@%
  #2\box\z@}}}
\gdef\leftnote#1{\n@te{\hss#1\quad}{}}
\gdef\rightnote#1{\n@te{\quad\kern-\leftskip#1\hss}{\moveright\hsize}}
\gdef\?{\FN@\qumark}
\gdef\qumark{\ifx\next"\DN@"##1"{\leftnote{\rm##1}}\else
 \DN@{\leftnote{\rm??}}\fi{\rm??}\next@}}
\newtheorem{theorem}{Theorem}[section]
\newtheorem{thm}[theorem]{Theorem}
\newtheorem{prop}[theorem]{Proposition}
\newtheorem{lem}[theorem]{Lemma}
\newtheorem{cor}[theorem]{Corollary}

\newtheorem{claim}[theorem]{Claim}

\newtheorem{questions}[theorem]{Questions}

\theoremstyle{definition}
\newtheorem{defn}[theorem]{Definition}

\newtheorem{remark}[theorem]{Remark}
\newtheorem{example}[theorem]{Example}

\numberwithin{equation}{section}

\date{March 4, 2010}
\begin{document}
%%%%%%%%%%%%%%%%%%%%%%%%%%%%%%%%%%%%%%%%%%

\newcommand{\Symp}{\mbox{\boldmath$\rm Sp$}}
\newcommand{\g}{\mathfrak{g}}
\newcommand{\el}{\mathfrak{l}}
\newcommand{\lt}{\mathfrak{t}}
\newcommand{\ls}{\mathfrak{s}}
\newcommand{\lc}{\mathfrak{c}}
\newcommand{\lu}{\mathfrak{u}}
\newcommand{\lr}{\mathfrak{r}}
\newcommand{\pr}{\operatorname{pr}}
\newcommand{\Hom}{\operatorname{Hom}}
\newcommand{\Rad}{\operatorname{Rad}}
\newcommand{\sign}{\operatorname{sign}}

\newcommand{\ve}{{\varepsilon}}
\newcommand{\vp}{{\varpi}}

\newcommand{\kbar}{\overline k}

\newcommand{\sdp}{\mathbin{{>}\!{\triangleleft}}} % <--- semidirect product
\newcommand{\Alt}{\operatorname{A}}   % < ---- the alternating group
\newcommand{\GL}{\operatorname{GL}}
\newcommand{\PGL}{\operatorname{PGL}}
\newcommand{\SL}{\operatorname{SL}}
\newcommand{\SO}{\operatorname{SO}}
\newcommand{\Ad}{\operatorname{Ad}}
\newcommand{\ad}{\operatorname{ad}}

\newcommand{\rank}{\operatorname{rank}}
\newcommand{\Aut}{\operatorname{Aut}}
\newcommand{\Char}{\operatorname{\rm char\,}} %% \char is already a command
\newcommand{\Gal}{\operatorname{Gal}}
\newcommand{\galois}{\Gal}
\newcommand{\rto}{\dasharrow}
\newcommand{\M}{\operatorname{M}}        % <--- matrix algebra
\newcommand{\ord}{\mathop{\rm ord}\nolimits}
\newcommand{\Sym}{{\operatorname{S}}}    % <--- the symmetric group
\newcommand{\tr}{\operatorname{\rm tr}}
\newcommand{\trace}{\tr}

\newcommand{\Res}{\operatorname{Res}}
\newcommand{\Sha}{\mbox{\rus{\fontsize{11}{11pt}\selectfont{SH}}}}
\newcommand{\G}{\mathcal{G}}
\renewcommand{\H}{\mathcal{H}}
\newcommand{\gen}[1]{\langle{#1}\rangle}
\renewcommand{\O}{\mathcal{O}}
\newcommand{\C}{\mathcal{C}}
\newcommand{\Ind}{\operatorname{Ind}}
\newcommand{\End}{\operatorname{End}}
\newcommand{\Spin}{\mbox{\boldmath$\rm Spin$}}
\newcommand{\T}{\mathbf G}
\newcommand{\GT}{\mbox{\boldmath$\rm T$}}
\newcommand{\Inf}{\operatorname{Inf}}
\newcommand{\Tor}{\operatorname{Tor}}
\newcommand{\m}{\mbox{\boldmath$\mu$}}

\newcommand{\A}{{\sf A}}
\newcommand{\D}{{\sf D}}
\newcommand{\Lbd}{{\sf \Lambda}}

\newcommand{\Id}{\operatorname{Id}}
\newcommand{\id}{\operatorname{id}}
\newcommand{\Nrd}{\operatorname{Nrd}}
\newcommand{\Trd}{\operatorname{Trd}}

\newcommand{\bbA}{{\mathbb A}}
\newcommand{\bbG}{{\mathbb G}}
\newcommand{\bbC}{{\mathbb C}}
\newcommand{\bbZ}{{\mathbb Z}}
\newcommand{\bbP}{{\mathbb P}}
\newcommand{\bbQ}{{\mathbb Q}}
\newcommand{\bbF}{{\mathbb F}}
\newcommand{\bbR}{{\mathbb R}}

\newcommand{\ssetminus}{\! \setminus \!}

\newcommand{\Ker}{\operatorname{Ker}}
\newcommand{\Spec}{\operatorname{Spec}}
\newcommand{\PGLn}{{\operatorname{PGL}_n}}
\newcommand{\PGLp}{{\operatorname{PGL}_p}}
\newcommand{\Sympl}{{\operatorname{Sp}}}
\newcommand{\Stab}{\operatorname{Stab}}
\newcommand{\Span}{\operatorname{Span}}
\newcommand{\diag}{\operatorname{diag}}
\newcommand{\Galois}{\Gal}
\newcommand{\Mat}{{\operatorname{M}}}
\newcommand{\Mn}{\Mat_n}
\newcommand{\Ima}{\operatorname{Im}}
\newcommand{\Int}{\operatorname{Int}}
\newcommand{\trdeg}{\operatorname{trdeg}}
\newcommand{\Tr}{\operatorname{Tr}}
\newcommand{\N}{\operatorname{N}}
\newcommand{\sln}{\operatorname{sl}_n}
\newcommand{\cal}{\mathcal}
\newcommand{\Lie}{\operatorname{Lie}}
\newcommand{\Ass}{\operatorname{Ass}}

\newcommand{\Br}{\operatorname{Br}}
\newcommand{\Pic}{\operatorname{Pic}}
\newcommand{\Div}{\operatorname{Div}}
\newcommand{\Brnr}{\operatorname{Br}_{\text{\rm{nr}}}}
\newcommand{\rk}{\operatorname{rk}}
\newcommand{\Isom}{\operatorname{Isom}}

\newcommand{\alp}{\alpha}
\newcommand{\eps}{\varepsilon}

\newcommand{\VG}{V{}_{{}^{\overset{}

\newcommand{\riso}{\hskip1mm {\buildrel \cong \over \rightarrow} \hskip1mm}
\newcommand{\liso}{\hskip1mm {\buildrel \cong \over \leftarrow} \hskip1mm}

{\centerdot}}}
  \hskip .5mm G\hskip
-3.9mm^{\overset{\centerdot}{}} \hskip
-1.37mm{}^{\underset{\centerdot}{}}\hskip 3.5mm{} }

\newcommand{\XH}{X{}_{{}^{\overset{}
{\centerdot}}}
  \hskip .5mm H\hskip
%%-3.9mm^{\overset{\centerdot}{}} \hskip
%%-1.37mm{}^{\underset{\centerdot}{}}\hskip 3.5mm{} }
-3.9mm^{\overset{\centerdot}{}} \hskip
-1.58mm{}^{\underset{\centerdot}{}}\hskip 3.5mm{} }

%\newcommand{\ker}{\operatorname{ker}}

%% CT command
\def\oi{\hskip1mm {\buildrel \cong \over \rightarrow} \hskip1mm}
\def\oii{\overset{%%\cong
\cong}\to}

%%\vskip -4mm

\title[Purity over the invariants of the adjoint action]{Is
the function field of a reductive Lie algebra\\
purely transcendental over\\ the field of  invariants for the
adjoint action?}

\author{Jean-Louis Colliot-Th\'el\`ene}
\address{C.N.R.S.,
UMR 8628, Math\'ematiques, B\^atiment 425, Universit\'e Paris-Sud,
F-91405 Orsay, France} \email{jlct@math.u-psud.fr}

\author{Boris Kunyavski\u\i }
\address{Department of Mathematics,
Bar-Ilan University, 52900 Ramat Gan, Israel}
\email{kunyav@macs.biu.ac.il}

\author{\break Vladimir L. Popov}
\address{Steklov Mathematical Institute,
Russian Academy of Sciences, Gubkina 8, Moscow 119991, Russia}
\email{popovvl@orc.ru}

\author{Zinovy Reichstein}
\address{Department of Mathematics, University of British Columbia,
       Vancouver, BC V6T 1Z2, Canada}
\email{reichst@math.ubc.ca}

\dedicatory{\rus Valentinu Evgenp1evichu Voskresenskomu, kollege i
uchitelyu,\\ s uvazheniem i blagodarnostp1yu}

\renewcommand{\baselinestretch}{1.0}

\begin{abstract} Let $k$ be a field of characteristic zero,
let $G$ be a    connected reductive algebraic group over $k$ and let
$\g$ be its Lie algebra. Let $k(G)$, respectively,\;$k(\g)$, be the
field of $k$-rational functions on $G$, respectively,\;$\g$. The
conjugation action of $G$ on itself induces the adjoint action of
$G$ on $\g$. We investigate the question   whether or not the field
extensions $k(G)/k(G)^G$ and $k(\g)/k(\g)^G$ are purely
transcendental. We show that the answer is the same for
$k(G)/k(G)^G$ and $k(\g)/k(\g)^G$, and reduce the problem to the
case where $G$ is simple. For simple   groups we show that the
answer is positive if $G$ is split of type ${\sf A}_{n}$ or ${\sf
C}_n$, and negative for groups of other types, except possibly ${\sf
G}_{2}$. A key ingredient in the proof  of the negative result is a
recent formula for the unramified Brauer group of a homogeneous
space with connected stabilizers.

As a byproduct of our investigation we give an affirmative answer to
a question of Grothendieck about the existence of a rational section
of the categorical quotient morphism for the conjugating action of
$G$ on itself.

The results and methods of this paper have played an
important part in recent A. Premet's negative solution ({\tt arxiv:0907.2500})
of the Gelfand--Kirillov conjecture for finite-dimensional simple Lie
algebras of every type, other than ${\sf A}_n$, ${\sf C}_n$, and
${\sf G}_2$.
\end{abstract}



\maketitle

\begin{quote}{\fontsize{8pt}{3.3mm}
\selectfont\tableofcontents}
\end{quote}

%%%%%%%%%%%%%%%%%%%%%%%%
%%%%%%%%%%%%%%%%%%%%%%%%

\section*{Introduction}
\label{sect1}

A field extension $E/F$ is called {\em pure} (or {\em purely
transcendental} or {\em rational}\,) if $E$ is generated over $F$ by
a finite collection of algebraically independent elements. A field
extension $E/F$ is called {\em stably pure} (or {\em stably
rational}\,) if $E$ is contained in a field $L$ which is pure over
both $F$ and $E$. Finally, we shall say that $E/F$ is {\em
unirational} if $E$ is contained in a field $L$ which is pure over
$F$. In summary,
\[ \xymatrix@R=4mm{
L   \ar@{-}[d]_{\rm pure} \\
E   \ar@{-}[d]_{\text{stably pure}} \\
F \ar@{-}@/_1.5pc/[uu]_{\rm pure}} \quad \xymatrix@R=4mm{
L   \ar@{-}[d] \\
E   \ar@{-}[d]_{\rm unirational} \\
\, \, F \ar@{-}@/_1.5pc/[uu]_{\rm pure} \, .} \]

Let $k$ be a field. Unless otherwise mentioned, we assume that ${\rm
char}(k)=0$. This is in particular a standing assumption in this
section.

Let $G$ be a  connected reductive algebraic group  over $k$. Let $V$
be a finite dimensional $k$-vector space and let $G \hookrightarrow
\GL(V)$ be an algebraic group embedding over $k$. Let $k(V)$ denote
the field of $k$-rational functions on $V$ and $k(V)^G$ the subfield
of $G$-invariants in $k(V)$. It is natural to ask whether
$k(V)/k(V)^G$ is pure (or stably pure).

This question may be viewed as a birational counterpart of the
classical problem of freeness of the module of (regular) covariants,
i.e., the $k[V]^G$-module $k[V]$; cf.~\cite[Sects.\;3 and
8]{popov-vinberg}. (Here $k[V]$ is the algebra of $k$-regular
functions on $V$ and $k[V]^G$ is the subalgebra of its $G$-invariant
elements.) The question of rationality of $k(V)$ over $k(V)^G$ also
comes up in connection with counterexamples to the Gelfand--Kirillov
conjecture; see~\cite{AOVdB} and the paragraph of this introduction
right after the statement of Theorem \ref{thm1.1}.

Recall that a connected reductive group $G$ is called {\it split} if
there exists a Borel subgroup $B$ of $G$ defined over $k$ and a
maximal torus in $B$ is split.

If $G$ is split and the $G$-action on $V$ is {\it generically free},
i.e., the $G$-stabilizers  of the points of a dense open set of $V$
are trivial, then the following conditions are equivalent:
\begin{enumerate}
\item[(i)] the extension $k(V)/k(V)^G$ is pure;
\item[(ii)] the extension $k(V)/k(V)^G$ is
unirational; \item[(iii)] the group $G$ is a ``special
group''.\end{enumerate} Over an algebraically closed field, special
groups were defined by Serre~\cite{serre} and classified by him and
Grothendieck~\cite{grothendieck} in the 1950s; cf.~\S\ref{sect2.3}.
The equivalence of these conditions follows from
Corollary~\ref{cor.versal-c} below; see also
Lemma~\ref{pointunirat}.

The purity problem for $k(V)/k(V)^G$ is thus primarily of interest
in the case where the $G$-action on $V$ is faithful but not
generically free. For $k$ algebraically closed, such  actions have
been extensively studied and even classified, under the assumption
that either the group $G$ or the $G$-module $V$ is simple; for
details and further references, see
\cite[Sect.\;7.3]{popov-vinberg}.

For these $G$-modules, purity  for $k(V)/k(V)^G$ is known in some
special cases. For instance, one can show that this is the case if
$k[V]^G$ is generated by a quadratic form.  In \cite[Appendix
A]{AOVdB} one can find a sketch of a proof that $k(V)/k(V)^G$ is
pure if $G = \SL_2$ and $\dim(V) = 4$ or $5$.

Theorem~\ref{thm1}(b) below (with $G$ adjoint) gives the first known
examples of  a connected linear algebraic group $G$ over an
algebraically closed field $k$ with  a  faithful but not generically
free $G$-module $V$ such that $k(V)$ is not pure (and not even
stably pure) over $k(V)^G$.

Let   $\g$ be the Lie algebra of $G$.  The homomorphism $\Int\colon
G \to \Aut(G)$ sending  $g \in G$ to    the map $\Int(g) \colon G
\to G$, $ x \mapsto gxg^{-1}$, determines the {\it conjugation
action} of $G$ on itself,  $G \times G \to G,$ sending $(g,x)$ to
$\Int(g)(x)$.  The differential of $\Int(g)$ at the identity is the
linear map $\Ad(g) \colon\g \to\g$. This defines an action of $G$ on
$\g$, called the {\it adjoint action}. As usual, we will de\-no\-te
the fields of $k$-rational functions on $G$, respectively,\;$\g$, by
$k(G)$, respectively,\;$k(\g)$, and the fields of invariant
$k$-rational functions for the conjugation action,
respectively,\;the adjoint action, by $k(G)^G$,
respectively,\;$k(\g)^G$.

The purpose of this paper is to address the following purity
questions.

\begin{questions} \label{q1} Let $G$ be a  connected reductive
group over $k$ and let $\g$ be its Lie algebra.

\begin{enumerate}

\item[\rm(a)] Is the field extension $k(\g)/k(\g)^G$
pure{\rm?} stably pure{\rm?}

\item[\rm(b)] Is the field extension $k(G)/k(G)^G$
pure{\rm?} stably pure{\rm?}
\end{enumerate}
\end{questions}

It is worth mentioning here two other natural purity questions
arising in this situation, namely, that of the purity of  $k(\g)^G$
and $k(G)^G$ over $k$. They are, however, not directly related to
the questions we are asking and
 both
 have affirmative answers:  for $k(\g)^G$  this is proved
 in
\cite{kostant} and, for $k(G)^G\!,$ in  \cite{steinberg1} (for
simply connected semisimple $G$)
 and
\cite{popov5} (for the general case).

The main case of interest for us is that of {\rm split} groups, but
some of  our results hold for arbitrary reductive groups.

We shall give a nearly complete answer to Questions~\ref{q1} for
split groups, in particular, when $k$ is algebraically closed. Our
results can be summarized as follows.

\vskip 2mm

(i) (Corollary~\ref{cor.steinberg2}) {\it Let $G$ be a connected
split  reductive group over $k$. Then the field extensions
$k(\g)/k(\g)^G$ and $k(G)/k(G)^G$ are unirational.}

\vskip 2mm

This is closely related to Theorem~\ref{thm.steinberg0} below.

\vskip 2mm (ii) (Theorem~\ref{thm.reduction}) {\it For a given
connected reductive group $G$ over $k$, the answers to Questi\-ons
{\rm\ref{q1}(a)} and {\rm(b)} are the same.}

\vskip 2mm (iii) (Proposition~\ref{prop.reductive}) {\it For a
connected reductive  group $G$ over $k$ and a central $k$-subgroup
$Z$ of $G$,
  the answers to
  Questions~{\rm\ref{q1}}
  for $G/Z$ are the same as for $G$.}

\vskip 2mm

Taking $Z$ to be the radical of $G$, we thus reduce
Questions~\ref{q1} to the case where $G$ is semisimple. We shall
further reduce them to the case where $G$ is simple as follows.
Recall that a semisimple group $G$ is called {\it simple} if its Lie
algebra is a simple Lie algebra. Its centre is then finite but need
not be trivial. In the literature such a group is sometimes referred
to as an {\it almost simple} group.

\vskip 2mm (iv) (Proposition~\ref{prop.semisimple}) {\it Suppose
that $G$ is  connected, semisimple, and split. Denote the simple
components of the simply connected cover of   $\;G$ by $G_1, \dots,
G_n$. Let $\g_{i}$ denote the Lie algebra of $\,G_{i}$. Then the
following properties are equivalent:
\begin{enumerate}
\item[\rm(a)]
$k(\g)/k(\g)^G$ is stably pure;
\item[\rm(b)] $k(\g_i)/k(\g_i)^{G_i}$ is stably pure for every $i
= 1, \dots, n$.
\end{enumerate}
Similarly, the following properties are equivalent:
\begin{enumerate}
\item[\rm(a)]
$k(G)/k(G)^G$ is stably pure;
 \item[\rm(b)]
 $k(G_{i})/k(G_{i})^{G_{i}}$
is stably pure for every $i = 1, \dots, n$.
\end{enumerate}
}

\vskip 1.5mm If we replace ``stably pure" by ``pure", we can still
show that the field extension $k(\g)/k(\g)^G$ (respectively,
$k(G)/k(G)^G$) is pure if each $k(\g_{i})/k(\g_{i})^{G_{i}}$
(respectively, each $k(G_{i})/k(G_{i})^{G_{i}}$) is pure, but we do
not know whether or not the converse holds.

Finally, in the case where $G$ is simple we  prove the following
theorem.

\begin{thm} \label{thm1}
Let $G$ be a connected, simple algebraic group  over $k$
 and  let $\g$ be its Lie algebra. Then the field
extensions $k(G)/k(G)^G$ and $k(\g)/k(\g)^G$ are

\begin{enumerate}
\item[\rm(a)] pure, if $G$ is split of type ${\sf
A}_n$    or ${\sf C}_n$;

\item[\rm(b)] not stably pure if  $G$ is not of type
${\sf A}_n$, ${\sf C}_n$, or
  ${\sf G}_{2}$.
  \end{enumerate}
\end{thm}

To prove Theorem~\ref{thm1}, we show that the two equivalent
Questions~\ref{q1}(a) and (b) are equivalent to the question of
(stable) $K_{\rm gen}$-rationality of the homogeneous space
$G_{K_{\rm gen}}/T_{\rm gen}$, where $T_{\rm gen}$ is the generic
torus of $G$, defined over the field $ K_{\rm gen}$; see
Theorem~\ref{thm.reduction}. (For the definition of $T_{\rm gen}$
and $K_{\rm gen}$ see \S\ref{thegenerictorus}.) We then address this
rationality problem for $G_{K_{\rm gen}}/T_{\rm gen}$
  by using the main result of \cite{col-k}, which gives
a formula for the unramified Brauer group of a homogeneous space
with connected stabilizers; see~\S\ref{sectG/T}. This allows us to
prove Theorem~\ref{thm1}(b) in~\S\ref{sect.non-rationality} by
showing that if $\,G$ is not of type ${\sf A}_n$, ${\sf C}_n$, or
${\sf G}_2$, then the unramified Brauer group of $G_{K_{\rm
gen}}/T_{\rm gen}$ is nontrivial over some field extension of
$K_{\rm gen}$. This approach also yields Theorem~\ref{thm1}(a), with
``pure" replaced by ``stably pure" (Proposition
\ref{stablerationalityAC}). The proof of the purity assertion in
part (a) requires additional arguments, which are carried out
in~\S\ref{sect.rat}.

A novel feature of our approach is a systematic use of the notions
of $(G,S)$-fibration and versal $(G, S)$-fibration, generalizing the
well known notions of $G$-torsor and  versal $G$-torsor; cf.,
e.g.,~\cite[Sect.\;1.5]{gms}. Here $S$ is a $k$-subgroup of $G$. For
details we refer the reader to Sections \ref{quotientsGSfib} and
\ref{sect.versal}.

As a byproduct of our investigations we obtain the following two
results which are not directly related to Questions~\ref{q1} but
are, in our opinion, of independent interest. Recall that a
connected reductive group over a field $k$ is called {\it
quasisplit} if it has a Borel subgroup defined over $k$.

\begin{thm}
[Corollary~\ref{cor.steinberg2}(a)] \label{thm.steinberg0} Let $G$
be a connected quasisplit
 reductive group over $k$.
Then the categorical quotient map $G \to G \catq G$ for the
conjugation action has a rational section.
\end{thm}

In the classical case $G= {\rm SL}_{n}$ such a section is given by
the companion matrices. The existence of a regular section for an
arbitrary connected, split, semisimple, simply connected  group $G$
is a theorem of Steinberg~\cite[Theorems 1.4 and 1.6]{steinberg1}.
In a letter to Serre, dated January 15, 1969, Grothendieck asked
whether or not a rational section exists if $G$ is not assumed to be
simply connected; see~\cite[p.\,240]{G-S}\footnote{The exact quote
is as follows: {`Le th\'eor\`eme de Steinberg [\dots ] est-il vrai
uniquement pour le [groupe] simplement connexe [\dots ]~? Que se
passe-t-il par exemple pour ${\rm GP}(1)$~? Y a-t-il une section
rationnelle de $G$ sur $I(G)$ (``invariants'') dans ce cas~?'} }.
Theorem~\ref{thm.steinberg0} answers this question in the
affirmative. The example specifically mentioned by Grothendieck is
$\PGL_2$, or $\operatorname{GP}(1)$, in his notation; an explicit
rational section in this case is constructed in
Remark~\ref{rem.cayley}. Our proof of Theorem~\ref{thm.steinberg0}
does not use Steinberg's result, but it uses Kostant's result on the
existence of sections in the Lie algebra case (\cite{kostant},~
\cite{kottwitz2}).

If $k$ is algebraically closed and $G$ is not simply connected, then
by \cite{popov5} there is no regular section of the categorical
quotient map $G \to G \catq G$.

\begin{thm}
[see Propositions~\ref{prop5.3}
and~\ref{prop.weight-lattice}]\label{thm1.1} Let $G$ be a connected,
split,   simple, simply connected algebraic group defined over $k$
and let $W$ be the Weyl group of $\,G$. The weight lattice $P(G)$ of
$\,G$  fits into an exact sequence of $\,W$-lattices\vskip 1.3mm
$$0 \to P_{2} \to P_{1} \to P(G) \to 0 \, , $$
\vskip 3.3mm\noindent with $P_{1}$ and $P_{2}$ permutation, if and
only if $G$ is of type ${\sf A}_{n}$, ${\sf C}_{n}$, or ${\sf G}_2$.
\end{thm}

\vskip 2mm

Recently the results and methods of this paper have played an
important part in A.~Premet's (negative) solution of the
Gelfand--Kirillov conjecture for finite-dimensional simple Lie
algebras of every type, other than ${\sf A}_n$, ${\sf C}_n$, and
${\sf G}_2$. The Gelfand--Kirillov conjecture is known to be true
for Lie algebras of type ${\sf A}_n$ and is still open
 for the types ${\sf C}_n$ and ${\sf G}_2$.
For details, including background material on the Gelfand--Kirillov
conjecture, we refer the reader to~\cite{premet}.

\vskip 2mm

This paper is dedicated to Valentin Evgen'yevich Voskresenski\u{\i},
who turned 80 in 2007. Professor Voskresenski\u{\i}'s work (see
 \cite{voskresenskii-rus, voskresenskii}) was
the starting point for many of the methods and ideas used in the
present paper.

 \vskip 3mm

 {\bf Some terminology}

 \vskip 1mm

By definition, a $k$-variety is a separated $k$-scheme of finite
type. If $X$ is a $k$-variety, it is naturally equipped with its
structure morphism $X \to \Spec k$. As a consequence, any Zariski
open set $U \subset X$ is naturally a $k$-variety.

The fibre product over $\Spec k$ of two $k$-schemes $X$ and $Y$ is
denoted $X\times_{k}Y$, or simply  $X \times Y$ when the context is
clear.

The set   of $k$-rational points of a $k$-variety $X$ is defined by
$X(k)={\rm Mor}_{k}(\Spec k, X)$.

An algebraic group over $k$, sometimes simply called a $k$-group,
is a $k$-variety equipped with a structure of algebraic group over
$k$. In other words,
  there is a multiplication morphism $G \times_{k}G \to G$ and a neutral element
in $G(k)$ satisfying the usual properties. In other terms, it is a
$k$-group scheme of finite type.

The ring $k[X]$ is the ring of global sections of the sheaf
${\mathcal O}_{X}$ over the $k$-variety $X$. The group
$k[X]^{\times}$ is the group of invertible elements of $k[X]$.

If $X$ is an integral (i.e., reduced and irreducible) $k$-variety,
we let $k(X)$ be the field of rational functions on $X$. This is the
direct limit of the fields of fractions of the $k$-algebras $k[U]$
for $U$ running through the dense open subsets $U$ of $X$.

When  we consider two $k$-varieties $X$ and $Y$, a $k$-morphism from
$X$ to $Y$ will sometimes simply be called a morphism or even a map.

Similarly, if $H$ and $G$ are algebraic groups  over $k$, if the
context is clear, a $k$-homomorphism of $k$-group schemes from $H$
to $G$ will sometimes simply be called a homomorphism, or even a
morphism.

For any field extension $K/k$, we may consider the $K$-variety
$X_{K} =X\times_{k}K$, where the latter expression is shorthand for
$X \times_{\Spec\hskip .15mm k} {\Spec}\, K$. We write
$K[X]=K[X_{K}]$. If the $K$-variety $X_{K}$ is integral, we let
$K(X)$ be the function field of $X_{K}$.

An integral $k$-variety $X$ is called {\it stably $k$-rational}  if
its function field $k(X)$ is stably rational over $k$ (see the first
paragraph of the Introduction) or, equivalently, if there exists a
$k$-birational isomorphism between $X \times_{k} \bbA_{k}^n$ and
$\bbA_{k}^m$ for some integers $n,m \geqslant 0$. If $n=0$, $X$ is
called {\it $k$-rational}.

%%%%%%%%%%%%%%%%%%%%%%%%
%%%%%%%%%%%%%%%%%%%%%%%%
%%%%%%%%%%%%%%%%%%%%%%%%

\section{Preliminaries on lattices, tori, and special groups}
\label{sect2}

For the details on the results of this section, see \cite{CTSa},
\cite{CTSb}, \cite{voskresenskii}, or \cite{lorenz}.

\subsection{\boldmath$\Gamma$-lattices}
\label{sect2.1}

Let $\Gamma$ be a finite group. A $\Gamma$-{\it lattice}   $M$ is a
free abelian group of finite type equipped with a
  homomorphism $\Gamma \to  {\rm Aut}(M)$.
 When the context is clear, we shall say lattice instead of $\Gamma$-lattice.

In this subsection we recall some basic properties of such lattices.

The {\it dual lattice} of a lattice $M$ is the lattice
$M^0=\Hom_{\bbZ}(M,\bbZ)$ where for $\gamma\in \Gamma$, $m\in M$ and
$\varphi \in M^0$, we have
  $(\gamma\cdot\varphi)(m)=\varphi(\gamma^{-1}\cdot m)$.

A {\it permutation lattice} is a lattice which has a $\bbZ$-module
basis whose elements   are permuted by $\Gamma$. The dual lattice of
a permutation lattice is a permutation lattice.

Any lattice $M$ may be realized as a sublattice of a permutation
lattice $P$ with torsion-free quotient $P/M$ (\cite{CTSb}  Lemma
0.6).

Two lattices $M_{1}$ and $M_{2}$ are called {\it stably equivalent}
if there exist permutation lattices $P_{1}$ and $P_{2}$ and an
isomorphism $M_{1} \oplus P_{1} \cong M_{2} \oplus P_{2}$.

A lattice $M$ is  called a {\it stably permutation lattice} if there
exist permutation lattices $P_{1}$ and $P_{2}$ and an isomorphism $M
\oplus P_{1} \cong P_{2}$.

A lattice $M$ is called {\it invertible} if there exists a lattice
$N$ such that $M \oplus N$ is a permutation lattice.

In these definitions one may replace $\Gamma$ by its image   in the
group of automorphisms of $M$. Because of this one may  give the
analogous definitions for $\Gamma$  a profinite group with a
continuous and discrete action.

  Let  $M$ be a  $\Gamma$-lattice. For any integer $i\geqslant 0$ one writes
$$\Sha^{i}_{\omega}(\Gamma,M)=
{\rm Ker} \Bigl[H^{i}(\Gamma,M) \to \prod_{\gamma \hskip1mm \in
\hskip1mm \Gamma} H^{i}(\langle \gamma \rangle,M)\Bigr].$$ For
$i=1,2$, this kernel only depends on the image of $\Gamma$ in the
group of automorphisms of $M$. So for these $i$ it is natural to
extend the above definition to the case where $\Gamma$ is a
profinite group and the action is continuous and discrete.

If $\Gamma$ is the absolute Galois group of a field $K$, one refers
to lattices as Galois lattices and one uses the notation
$\Sha^{i}_{\omega}(K,M)$.

If $M$ is a permutation lattice, then for any subgroup $\Gamma'$ of
$\Gamma$,
$$H^1(\Gamma',M)=0, \hskip3mm  H^1(\Gamma',M^0)=0.$$
 Moreover,
$$\Sha^{2}_{\omega}(\Gamma',M)=0, \hskip3mm \Sha^{2}_{\omega}(\Gamma',M^0)=0.$$

If there exists an exact sequence
$$ 0 \to P_{1 } \to P_{2} \to M \to 0$$
with $P_{1}$ and $P_{2}$ permutation lattices, then
$\Sha^{1}_{\omega}(\Gamma',M)=0$ for any subgroup $\Gamma'$ of
$\Gamma$.

\subsection{Tori}
\label{sect2.2}

Let $K$ be a field, let $K_{\rm s}$ be a separable closure of $K$,
and let $\Gamma$ denote the Galois group of $K_{\rm s}/K$.  A $
K$-torus  $T$ is an algebraic $K$-group which over an algebraic
field extension $L/K$ is isomorphic to a product of copies of the
multiplicative group $\bbG_{m,L}$. The field $L$ is then called {\it
a splitting field} for $T$. Inside $K_{\rm s}$ there is a smallest
splitting field for $T$, it is a finite Galois extension of $K$,
called {\it the splitting field} of $T$.

To any $K$-torus $T$ one may associate two $\Gamma$-lattices: its
(geometric) {\it character group}
$$T^*=\Hom_{K_{\rm s}\mbox{\rm \scriptsize -gr}}(T_{K_{\rm s}},\bbG_{m,K_{\rm s}} ) $$ and
its  (geometric) {\it cocharacter group}
$$T_{*}=\Hom_{K_{\rm s}\mbox{\rm \scriptsize
-gr}}(\bbG_{m,K_{\rm s}} ,  T_{K_{\rm s}}).$$ These two
$\Gamma$-lattices are dual of each other.

The association  $T \mapsto T_{*}$ defines an equivalence between
the category of $K$-tori and the category of $\Gamma$-lattices. The
association $T \mapsto T^*$ defines a duality (anti-equivalence)
between the category of $K$-tori and the category of
$\Gamma$-lattices.

The $K$-torus whose character group is  $T_{*}$ is denoted $T^0$ and
is called {\em the torus dual to $T$}.

A $K$-torus is called {\it quasitrivial}, respectively,\;{\it stably
quasitrivial}, if its character group, or equivalently its
cocharacter group, is a permutation lattice, respectively,\;is a
stably permutation lattice. A quasitrivial torus $T$ is
$K$-isomorphic to a product of tori of the shape
  $R_{L/K}\bbG_{m}$, i.e., Weil restriction of scalars of the multiplicative
  group
$\bbG_{m,L}$ from $L$ to $K$, where $L/K$ is a finite separable
  field extension. A quasitrivial $K$-torus is an open set
of an affine $K$-space, hence is $K$-rational.

By a theorem of Voskresenski\u{\i}, a  $K$-torus of dimension at
most 2 is $K$-rational (\cite[\S 2.4.9, Examples 6 and
7]{voskresenskii}). This implies the following property. For any
$\Gamma$-lattice $M$ which is a direct sum of lattices of rank at
most 2, there exist exact sequences
\begin{gather*}
0 \to P_{2} \to P_{1} \to M \to 0,\\
\end{gather*}
where the $P_{i}$'s and $P'_{i}$'s are permutation lattices.

\subsection{Special groups}
\label{sect2.3}

  Let $K$ be a field of characteristic zero. Recall from \cite{serre}  that
an algebraic group $G $ over $K$ is called {\em special} if for any
field extension $L/K$, the Galois cohomology set $H^1(L,G)$ is
reduced to one point. In other words, $G$ is special if every
principal homogeneous space under $G$ over a field containing $K$ is
trivial. Such a group is automatically linear and
connected~\cite{serre}. An extension of a special group by a special
group is a special group. A unipotent group is special. A
quasitrivial torus is special. So is a direct factor of a
quasitrivial torus (such a $K$-torus need not be stably
$K$-rational). If $K$ is algebraically closed and $G$ is semisimple,
then $G$ is special if and only if it is isomorphic to a direct
product
\[ \SL_{n_1} \times \dots \times \SL_{n_r} \times
\Sympl_{2m_1} \times \dots \times \Sympl_{2m_s} \] for some integers
$r, s, n_1, \dots, n_r, m_1, \ldots, m_s$. That such groups are
special is proved in~\cite{serre}, that only these are is proved
in~\cite{grothendieck}.

%%%%%%%%%%%%%%%%%%%%%%%%
%%%%%%%%%%%%%%%%%%%%%%%%
%%%%%%%%%%%%%%%%%%%%%%%%

\section{
Quotients, $(G,S)$-fibrations, and
$(G,S)$-varieties}\label{quotientsGSfib}

We recall that $k$ is a field of characteristic zero. Let $\kbar$ be
an algebraic closure of $k$,  and let $G$ be a (not necessarily
connected) linear algebraic group over $k$.

\subsection{Geometric quotients}
\label{sect.quotients}

 Let  us recall some standard definitions and facts.
For references, see \cite[Sect.\,II.6]{borel},
\cite[Sect.\,12]{humphreys}, \cite[Sects.\,5.5,\,12.2]{springer},
\cite[Sect.\,4]{popov-vinberg},
 and \cite{chili}.

 Let $X$ be a $k$-variety endowed with an  action of the $k$-group  $G$.
 A {\it geometric quotient} of $X$ by $G$ is a pair $(Y, \pi)$,
where $Y$ is a $k$-variety, called the {\it quotient space}, and
$\pi\colon X\to Y$ is a $k$-morphism, called the {\it quotient map},
such that
\begin{enumerate}

 \item[(i)] $\pi$ is an open {\it orbit map}, i.e., constant on $G$-orbits and  induces  a bijection
 of $X(\kbar)/G(\kbar)$ with $Y(\kbar)$;

 \item[(ii)] for every open subset $V$ of $Y$, the natural  homomorphism
 $\pi^*\colon k[V] \to k[\pi^{-1}(V)]^G$ is an isomorphism.
\end{enumerate}

 If such a pair $(Y, \pi)$ exists, it has the {\it universal mapping property}, i.e., for every $k$-morphism $\alpha\colon X\to Z$ constant on the fibres of $\pi$, there is a unique $k$-morphism $\beta\colon Y\to Z$
such that $\alpha=\beta\circ \pi$. In particular, $(Y, \pi)$ is
unique up to a unique $G$-equivariant isomorphism of total spaces
commuting with quotient maps. Given this, we shall denote $Y$ by
$X/G$.

\begin{remark} \label{rem.separate}
 Let $X$ and hence $X/G$ be geometrically integral. Being constant on $G$-orbits,  $\pi$
induces an embedding of fields $\pi^* \colon  k(X/G) \hookrightarrow
k(X)^G$. Conditions (i) and (ii) imply that, in fact, the latter is
an isomorphism of fields
 $\pi^*\colon  k(X/G) \oii k(X)^G$,
 see, e.g.,\;\cite[II, 6.5]{borel} (this property holds for
 the ground field $k$ of arbitrary characteristic; for ${\rm char}(k)=0$, it follows already from condition (i),
see, e.g.,\;\cite[Lemma 2.1]{popov-vinberg}).
\end{remark}

 If $G$ acts on a
 reduced
 $k$-variety $Z$ whose irreducible components are open,
 $B$ is a normal $k$-variety, $\varrho\colon Z \to B$ is a $k$-morphism
 constant on $G$-orbits, and $\varrho$ induces a bijection of
 $Z(\kbar)/G(\kbar)$ with
 $B(\kbar)$, then $(B, \varrho)$ is
 the geometric quotient  of $Z$ by $G$; see \cite[Prop.\;6.6]{borel}.

\begin{example}\label{G/H}
 If  $H$ is a closed $k$-subgroup of $G$,
the action of $H$ on $G$ by right translation gives rise to a
geometric quotient ${\pi}^{}_{G, H}\colon G \to  G/H$ called the
quotient of $G$ by $H$. The group $G$ acts on $G/H$ by left
translation and, up to $G$-isomorphism, $G/H$ is uniquely defined
among the homogeneous spaces of $G$ by the corresponding universal
property, see, e.g., \cite[Sects.\,5.5, 12.2]{springer}
or~\cite[Sect.\,12]{humphreys}.
\end{example}

For any reduced $k$-variety $X$ endowed with a $G$-action,
 a theorem of Rosenlicht \cite{rosenlicht1}, \cite{rosenlicht2}
  (cf.\;also \cite[Sects.\;2.1--2.4]{popov-vinberg},
 \cite[Satz 2.2]{springer2},
 \cite[Prop.\;4.7]{thomason})
ensures that there exist a $G$-invariant dense open subset $\;U$ of
$X$, a $k$-variety $Y$, and a smooth $k$-morphism $\alpha\colon U
\to Y$ such that
 $(Y, \alpha)$ is the geometric quotient of $\;U$ by $G$.

\subsection{\boldmath $(G,S)$-fibrations}\label{GSfibrations}
Consider the category $\mathcal M_G$ whose objects are $k$-morphisms
of $k$-schemes $\pi\colon X\to Y$ such that $X$ is endowed with an
action of $G$ and $\pi$ is constant on $G$-orbits, and a morphism of
$\pi_1\colon X_1\to Y_1$ to $\pi_2\colon X_2\to Y_2$ is a
commutative diagram
\begin{equation}\label{di}
\begin{matrix}
\xymatrix{X_1\ar[r]^{\alpha}\ar[d]_{\pi_1}&X_2\ar[d]^{\pi_2}\\
Y_1\ar[r]^{\beta}&Y_2}
\end{matrix}\quad,
\end{equation}
where  $\alpha$ and $\beta$ are $k$-morphisms and $\alpha$ is
$G$-equivariant. The notion of composition of morphisms is clear. A
morphism as in \eqref{di} is an isomorphism if and only if $\alpha$
and $\beta$ are isomorphisms.

 Let $\pi\colon X\to Y$ be an object of $\mathcal M_G$ and let $\mu\colon Z\to Y$ be a
$k$-morphism of $k$-schemes. Then $G$ acts on $X\times_Y Z$ via $X$,
and the second projection $X\times_Y Z\to Z$ is an object of
$\mathcal M_G$. We say that
 it is obtained from $\pi$ by the {\it base change~$\mu$}.

 \begin{defn}\label{fib}
 Let $F$ be a $k$-scheme endowed with an action  of $G$ and let $\pi\colon X\to Y$ be an object of $\mathcal M_G$.
The morphism $\pi$ is called
\begin{enumerate}
\item[(i)]
{\it trivial} (\hskip -.3mm{\it over} $Y$) {\it with fibre $F$} if
there exists an isomorphism between $\pi$ and ${\rm pr}_2\colon
F\times_k Y\to Y$ where $G$ acts on $F\times_k Y$ via $F$;

\item[(ii)]
{\it fibration} (\hskip -.3mm{\it over} $Y$) {\it with fibre $F$} if
$\pi$ becomes trivial with fibre $F$ after a surjective \'etale base
change $\mu\colon Y'\to Y$. In this case, we say that $\pi$ {\it is
trivialized by} $\mu$.
\end{enumerate}
  \end{defn}

\begin{example}\label{torsor}
If $F=G$ with the $G$-action by left translation, then the notion of
fibration over $Y$ with fibre $G$ coincides  with that of $\,G$-{\it
torsor} over $Y$.
\end{example}

The following definition extends the definition of a $G$-torsor (the
latter corresponds to the case where $S$ is the trivial subgroup
$\{1\}$).

 \begin{defn}\label{GS}
Let $S$ be a closed $k$-subgroup of $G$. A fibration with fibre
$G/S$, where $G$ acts on $G/S$ by left translation, is called $(G,
S)$-fibration.
 \end{defn}

If $X$ is a $k$-scheme endowed with an action of $G$ and there is a
$(G,S)$-fibration $X\to Y$, then we say that $X$ {\it admits the
structure of a $(G,S)$-fibration} over $Y$.

\begin{remark} \label{S} Replacing in Definition \ref{fib}(ii)
``\'etale'' by ``smooth'', one obtains an equivalent defi\-nition.
This follows from the fact that a surjective smooth morphism $Y' \to
Y$ admits sections locally for the \'etale topology on $Y$: there
exists an  \'etale surjective morphism $Y''\to Y$ such that
$Y'\times_{Y}Y'' \to Y''$ has a section (\cite[17.16.3]{EGA4}).
\end{remark}

\begin{remark}
If the $k$-scheme $X$ admits the structure of a $(G,S)$-fibration,
then it admits the structure of a $(G,S')$-fibration for any
$k$-subgroup $S'$ of $G$ such that $S_{\kbar}$ and ${S'}\hskip
-1mm_{\kbar}$ are conjugate subgroups of $G_{\kbar}$. Such
$k$-groups $S$ and $S'$ need not be $k$-isomorphic.
\end{remark}

We list some immediate properties without proof.
\begin{prop}\label{properties}
 Let $\pi\colon X \to Y$ be a $(G,S)$-fibration.
Then the following properties hold:
\begin{enumerate}
\item[\rm(i)]
$\pi$ is a smooth surjective morphism;

\item[\rm(ii)]
a morphism obtained from $\pi$ by a base change is a $(G,
S)$-fibration.
\end{enumerate}
\noindent Assume that $X$ is a $k$-variety. Then

\begin{enumerate}
\item[\rm (iii)] the $G$-stabilizers of
points of $\;X(\overline k)$ are conjugate to the subgroup
$S_{\overline k}$ of $\;G_{\overline k}$;

\item[\rm(iv)]
$(Y, \pi)$ is the geometric quotient of $\;X$ by $G$. \qed
\end{enumerate}
\end{prop}

If a $k$-group $S$ acts on a $k$-variety $Z$, then the functor on
commutative $k$-algebras $A \mapsto Z(A)^{S(A)}$
 defines a sheaf for the fppf topology on $\Spec k$.
This functor is representable by a closed $k$-subvariety $Z^S$ of
$Z$ (see \cite{fogarty} and \cite{sga3new}).

Let $Y$ be a $k$-variety and let $Z\to Y$ be an object of $\mathcal
M_S$. If $Y'$ is a $k$-variety and $Y'\to Y$ is a $k$-morphism, then
the natural $Y'$-morphism $Z^S\times_{Y}Y' \to (Z\times_{Y}Y')^S$
is an isomorphism
\begin{equation}\label{isoS}
Z^S\times_{Y}Y' \overset{\cong}\to (Z\times_{Y}Y')^S.
\end{equation}

Let $S$ be a closed $k$-subgroup of a $k$-group $G$ and let $N$ be
the {\it normalizer} of $\,S$ in $G$. Assume that $X$ is a
$k$-variety endowed with an action of $G$. Then the subvariety $X^S$
is $N$-stable and, since $S$ acts trivially on it, the action of
$\;N$ on $X^S$ descends to an action of the group
\begin{equation}\label{N/S}
H:=N/S.
\end{equation}

Let $Y$ be a $k$-variety and let $\pi\colon X\to Y$ be an object of
$\mathcal M_G$. Put
$$\pi^S:=\pi|_{X^S}\colon X^S\to Y$$
and let $\pi'\colon X\times_Y X^S\to X^S$ be the morphism obtained
from $\pi$ by the base change $\pi^S$:
\begin{equation*}
\begin{matrix}
\xymatrix@C=6mm{X\times_Y X^S\ar[r] \ar[d]_{\pi'
}& X\ar[d]^\pi\\
X^S\ar[r]^{\pi^S}& Y}
\end{matrix}\quad .
\end{equation*}
Since $\pi^S$ is constant on $H$-orbits, $H$ acts on $X\times _Y
X^S$ via $X^S$. The actions of $G$ and $H$ on $X\times_Y X^S$
commute. Therefore $X\times_Y X^S$ is endowed with an action of
$G\times_k H$. The morphism $\pi'$ is $H$-equivariant and constant
on $G$-orbits.

The group $H$ also acts on $G/S$ by right multiplication. This
action and the action of $H$ on $X^S$ determine the $H$-action on
$(G/S)\times_k X^S$. It commutes with the $G$-action on
$(G/S)\times_k X^S$ via left translation of $\,G/S$. Therefore
$(G/S)\times_k X^S$ is endowed with an action of $\,G\times_k H$.
The second projection $(G/S)\times_k X^S\to X^S$ is $H$-equivariant
and constant on $G$-orbits.

The natural morphism $H\to G/S$ yields the basic isomorphism
\begin{equation}\label{H}
H\xrightarrow{\cong}(G/S)^S.
\end{equation}

\begin{prop}\label{GSfibration0}
For every $(G,S)$-fibration $\pi\colon X \to Y$ where $X$ and $Y$
are $k$-varieties, the fol\-lowing properties hold:
\begin{enumerate}
\item[\rm(i)] $\pi^S\colon X^S \to Y$
is an $ H$-torsor and every base change trivializing $\pi$
trivializes $\pi^S$ as well;

\item[\rm(ii)] for the $(G, S)$-fibration $X\times_Y X^S\to X^S$ obtained from $\pi$ by the base change $\pi^S$,
the $(G\times_k H)$-equi\-variant
    $X^S$\hskip -.6mm-map
\begin{gather*}\label{GSX}
\begin{gathered}
\varphi\colon (G/S) \times_{k} X^S \to X \times_{Y}X^S, \quad
(\widetilde g,x)\mapsto(g\cdot x,x),
\end{gathered}
\end{gather*}
where $\widetilde{g}=\pi^{\ }_{G, S}(g)$ {\rm(}see Example
{\rm\ref{G/H})}, is an isomorphism;

\item[\rm(iii)]
    the morphism
    \begin{equation}\label{pp}
    {\rm pr_1}\circ \varphi \colon (G/S) \times_{k} X^S
\to X, \quad (\widetilde g, x)\mapsto g\cdot x,\end{equation} is an
$H$-torsor.
\end{enumerate}
\end{prop}

\begin{proof} (i) Let $\pi$ be trivialized by a (surjective \'etale) base change $\mu\colon Y'\to Y,$ i.e.,
there is a $G$-equivariant $Y'$-isomorphism $(G/S)\times_k Y'\to
X\times_Y Y'$. By \eqref{isoS} and \eqref{H} we then have the
$H$-equivariant $Y'$-isomorphisms $H\times_k Y'=(G/S)^S\times_k
Y'\!\to\! (X\times_Y Y')^S\!\to\! X^S\times_Y Y'$. Hence, $\pi^S$ is
an $H$-torsor which is trivialized by $\mu$. This proves~(i).

(ii) The morphism $\varphi$ is a $Y$-map with respect to the
compositions of the second projections with $\pi^S$. By
\cite[Vol.\,24, Prop.\,2.7.1(viii)]{EGA4} it is enough to prove the
claim for the morphism of varieties obtained by the base change
$\mu$ considered in the above proof of (i). By virtue of (i)  and
\eqref{N/S} this
 reduces the problem
to proving that the map  $$(G/S)\times_{k} (N/S) \to
(G/S)\times_{k}(N/S), \quad (\widetilde g, \widetilde n)\mapsto
(\widetilde{gn}, \widetilde n),$$ is an isomorphism. But this is
clear since $(\widetilde g, \widetilde n)\mapsto
(\widetilde{gn^{-1}}, \widetilde{n})$ is the inverse map. This
proves (ii).

(iii) By (i) and Proposition \ref{properties}(ii) the morphism
$X\times_Y X^S\to X$ obtained from $\pi^S$ by the base change $\pi$
is an $H$-torsor. Since
 $\varphi$ is an isomorphism, this proves~(iii).
\end{proof}

Let $C$ be an algebraic $k$-group. Consider a $C$-torsor
$$
\alpha\colon P\to Y
$$
over a $k$-variety $Y$. Let $F$ be a $k$-variety endowed with an
action of $C$. If every finite subset of $F$ is contained in an open
affine subset of $F$ (for instance, if $F$ is quasi-projective),
then for the natural action of $\,C$ on $F\times_k P$ determined by
the $C$-actions on $F$ and $P$, the geometric quotient exists; it is
usually denoted by
$$F\times^C\!P.$$
Moreover, the quotient map $F\times P \to F\times^C  P$ is actually
a $C$-torsor over $F\times^C  P$ (see
  \cite[Prop.\;2.12]{florence} and use Proposition \ref{properties}~(iv) above;
see also \cite[Sect.\;4.8]{popov-vinberg}).

Since the composition of morphisms $F\times_{k}P \xrightarrow{{\rm
pr}_2} P \xrightarrow{\alpha} Y$ is constant on $C$-orbits, by the
universal mapping property of geometric quotients this composition
factors through a $k$-morphism
$$\alpha^{}_{F}\colon F\times^CP \to Y.$$

Let $\mu\colon Y'\to Y$ be a surjective \'etale $k$-morphism such
that $\alpha$ becomes the trivial morphism ${\rm pr}_2\colon
C\times_k Y'\to Y'$ after the base change $\mu$. Then, after the
same base change, $\alpha^{}_F$ becomes  the morphism $$F\times^C
(C\times_k Y')=F\times_k Y'\xrightarrow{{\rm pr}_2} Y'.$$ Hence
$\alpha^{}_{F}$ is a fibration over $Y$ with fibre $F$.

Since the variety $G/S$ is quasi-projective (see
\cite[Theorem\;6.8]{borel}), this construction is applicable for
$C=H$ and $F=G/S$.

Given a $k$-variety $Y$,
 we now have two
constructions:
\begin{enumerate}
\item[$\cdot$] if
$\pi\colon X \to Y$ is a $(G,S)$-fibration, then $\pi^S\colon X^S
\to Y$ is an $H$-torsor;

\item[$\cdot$] if
$\alpha\colon P \to Y$ is an $H$-torsor,
 then
$\alpha_{G/S}\colon (G/S)\times^HP \to Y$ is a $(G,S)$-fibration.
\end{enumerate}

\begin{prop}\label{propGSfibration}
These two constructions are inverse to each other and they are
functorial in~$Y$.
\end{prop}
\begin{proof} Since by Proposition \ref{GSfibration0}(iii) morphism \eqref{pp} is an $H$-torsor, $X$ is the geometric quotient for the $H$-action on $(G/S)\times_k X^S$.
Hence, by the uniqueness of geometric quotient, there is a
$G$-equivariant isomorphism $(G/S)\times^HX^S \to X$.

Let $P\to Y$ be an $H$-torsor. The natural $H$-action on
$(G/S)\times_{k}P$ and the $G$-action
 via left translation of $G/S$ commute. From this and
\eqref{H} we deduce the isomorphisms
$$((G/S)\times^H\!P)^S
\xrightarrow{\cong}(G/S)^S\times^H\!P \xrightarrow{\cong}
H\times^H\!P \xrightarrow{\cong}P.$$

Functoriality in $Y$ is clear.
\end{proof}

\begin{remark} \label{rem.forms}
 The \'etale  \v{C}ech cohomology set $H^1(Y, H)$
 classifies $H$-torsors over $Y$. On the other hand,
 if ${\rm Aut}_G(G/S)$ is the algebraic $k$-group
 of $G$-equivariant automorphisms of $G/S$,
 then according to the general principle outlined at the
 beginning of~\cite[Section 3]{serre-gc} (see also the references there for
 a more rigorous treatment),
 the \'etale \v{C}ech cohomology set $H^1(Y, {\rm Aut}_G(G/S))$
 classifies $(G/S)$-fibrations over $Y$.
 Since the  $G$-action on $G/S$ by left translation commutes
 with the $H$-action  by right multiplication, we have
 an injection $H\hookrightarrow{\rm Aut}_G(G/S)$. It is well
 known (and easy to prove) that, in fact, $H={\rm Aut}_G(G/S)$.
  We thus get a bijection  between
 $H^1(Y,H)$ and $H^1(Y, {\rm Aut}_G(G/S))$, i.e., between $H$-torsors and
 $(G, S)$-fibrations over $Y$.
 Proposition \ref{propGSfibration} is an explicit version of this fact.
\end{remark}

\subsection{\boldmath $(G,S)$-varieties}
\label{sect.(G,S)-variety}
\begin{defn}\label{defGSvariety1}
Let  $S$ be a closed $k$-subgroup of $G$ and let $X$ be a
$k$-variety endowed with a $G$-action. We shall say that $X$ is a
{\it $(G,S)$-variety} if $X$ contains a dense open  $G$-stable
subset $U$ which admits the structure of a $(G,S)$-fibration $U \to
Y$. Generalizing a terminology introduced in \cite{berhuy-favi}, it
is convenient to call such an open subset $U$  a {\it friendly open
subset} of $X$ for the action of~$G$.
\end{defn}

If $X$ is geometrically integral and $U$ is a friendly  open subset
of $X$ with $(G,S)$-fibration $\pi\colon U \to Y$, then $\pi$
induces an iso\-morphism $\pi^*\colon k(Y) \oii k(X)^G$.

The following statement over an algebraically closed field has
previously appeared in various guises in the literature (see
\cite[2.7]{popov-vinberg} and \cite[1.7.5]{popov4}).

\begin{thm}\label{generalGS}
Let $X$ be a geometrically integral $k$-variety endowed with a
$G$-action. Let $S$ be a closed $k$-subgroup of $\;G$. Then the
following properties are equi\-valent:
\begin{enumerate}
\item[\rm (a)] $X$ is a $(G, S)$-variety;
\item[\rm (b)] $X$ contains a dense
open $G$-stable subset such that the $G_{\overline k}$-stabilizer of
each of its $\overline k$-po\-ints is conjugate to $S_{\overline
k}$.
\end{enumerate}
\end{thm}

\begin{proof}   That (a) implies (b) is clear. Let us assume (b).
One may replace $X$ by dense $G$-stable open subsets to successively
ensure that:

\smallskip

\hskip 1.8mm(i) The $G_{\overline k}$-stabilizer of every $\overline
k$-point of $X$ is conjugate to $S_{\overline k}$.

\hskip 1.1mm(ii) The $k$-variety $X$ is smooth over $k$.

(iii) (Rosenlicht's theorem, see Section \ref{sect.quotients}.)
There exist a geometrically integral $k$-variety $Y$ and a
$k$-morphism $\pi\colon X \to Y$ such that the pair $(Y,\pi)$ is the
geometric quotient of $X$ for the action of $G$.

\hskip .8mm(iv) The morphism $\pi$ is smooth. Indeed the generic
fibre of $\pi$ is regular, hence smooth since ${\rm char}(k)=0$. The
statement on $\pi$ can thus be achieved by replacing $Y$ by an open
set and $X$ by the inverse image of this open set.

 \hskip 1.9mm(v)
 There exists an open set $U$ of the reduced
variety
 $(X^S)_{\rm red}  \subset X^S $ such that the compositi\-on of maps $U
\hookrightarrow X^S \overset{\pi^S}\to Y$ is smooth.
 This follows from the surjectivity of the map $\pi^S\colon X^S \to Y$ on $\overline k$-points,
 itself a consequence of (i) and (iii).

 (vi) If we let $U \subset (X^S)_{\rm red}$ be the maximal open set such that the map $\pi|_U\colon U \to Y$
 is smooth, then this map
 is surjective. This is achieved by replacing $Y$
 in the previous statement by a dense open set contained in the image of
 the previous $U$, and replacing $X$ by the inverse image of this open set.

\smallskip

We now consider the following $G$-equivariant $k$-morphism of smooth
$k$-varieties:
$$\psi \colon (G/S) \times_{k} U \to X
\times_{Y}U, \quad (\widetilde g,u) \mapsto (gu,u)$$ (see the
notation in Proposition \ref{GSfibration0}(ii)). It is a
$U$-morphism with respect to the second projections. Since every
$G$-equivariant morphism $G/S\to G/S$ is bijective and $(Y, \pi)$ is
the geometric quotient, we deduce from (i) and (iii) that $\psi$ is
bijective on $\overline k$-points.
 As ${\rm char}(k)=0$, we then conclude
by Zariski's main theorem that $\psi$ is an isomorphism. Thus it is
proved that the morphism obtained from $\pi$ by the base change
$\pi|_U$ is trivial over $U$ with fibre $G/S$. Since $\pi|_U$ is
smooth, we now deduce from Remark \ref{S} that $\pi$ is a $(G,
S)$-fibration.
\end{proof}

Condition (b) of Theorem \ref{generalGS} gained much attention in
the literature (see \cite[\S7]{popov-vinberg} and references
therein). If (b) holds, one says that, for the action of $G$ on $X$,
there exists a {\it stabilizer $S$ in general position} or that
there exists a {\it principal orbit type} for $(G,X)$.  There are
actions, even of reductive groups, for which  a stabilizer in
general position does not exist (see \cite[7.1,
2.7]{popov-vinberg}). There are results ensuring its existence under
certain conditions or, equivalently (by Theorem \ref{generalGS}),
the existence of a structure of $(G, S)$-variety. Theorem
\ref{richardson1}  below is such a result.

Recall the following definition introduced in \cite{popov1}.

 \begin{defn}  The action of an algebraic $k$-group
$G$ on a $k$-variety  $X$ is called {\it stable} if there exists a
dense open subset $\;U$ of $X$ such that the $G$-orbit of every
point of $U(\overline k)$ is closed in $X_{\overline k}$.
 \end{defn}

\begin{thm} \label{richardson1}
Let $X$ be an affine geometrically integral $k$-variety with an
action of a reductive $k$-group $G$ such that $X(k)$ is Zariski
dense in $X$. Assume that either of the following conditions hold:
\begin{enumerate}
\item[\rm (i)] $X$ is smooth; or \item[\rm (ii)] the
$G$-action on $X$ is stable.
\end{enumerate}
Then there is a closed $k$-subgroup $S$ of $\,G$ such that $X$ is a
$(G, S)$-variety. In case {\rm (ii)} this sub\-gro\-up $S$ is
reductive.
\end{thm}

\begin{proof} If (i) holds, then
by Richardson's theorem~\cite[Prop.\,5.3]{richardson-stab}
(cf.\;also \cite[Cor.\;8]{luna}, \cite[Theorem 7.2]{popov-vinberg})
there is a closed ${\overline k}$-subgroup
 $R$ of  $G_{\overline k}$ such that
the $G_{\overline k}$-stabilizer of a general $\overline k$-point of
$X$ is conjugate to $R$. Since $X(k)$ is Zariski dense, $R$ can be
taken as $S_{\overline k}$, where $S$ is the stabilizer of a
$k$-point of $X$. Then property (b) from the
 statement of Theorem \ref{generalGS} holds, hence
  $X$ is a
 $(G, S)$-variety.

If (ii) holds, then the above subgroup $S$ still exists by
\cite[Sect.\;7.2, Cor.]{popov-vinberg}, so the same argument
applies.  As the general orbit is closed, it is affine, whence $S$
is reductive by Matsushima's criterion (\cite{matsushima},
\cite{onishchik}, cf.\,\cite{bialynickibirula}, \cite{luna}).
\end{proof}

\subsection{Categorical quotients}\label{cat}
For the definition of a categorical quotient we refer the reader
to~\cite[Def.\;0.5]{mumford}, \cite[6.16, 8.19]{borel},  and
~\cite[Sect.\;4.3]{popov-vinberg}. In this paper we shall only work
with categorical quotients for reductive group actions on affine
varieties, which are constructed as follows.

Let $A$ be a finitely generated $k$-algebra. Assume a reductive
$k$-group $G$ acts  on the $k$-variety $X= \Spec(A)$ (over $k$).
Then (cf.\;\cite[Theorem~1.1, Cor.\,1.2]{mumford})
 \begin{enumerate}
 \item[(i)]
 the ring $A^G$ is a finitely generated $k$-algebra;
\item[(ii)] the inclusion $A^G \hookrightarrow A$
induces a categorical quotient map $\pi\colon X \to \Spec(A^G) =:
X\catq G$;
 \item[(iii)]
 every geometric fibre of $\pi$ contains a unique closed orbit.
\end{enumerate}
As $G$-orbits are open in their closure, the latter property implies
that every geometric fibre of $\pi$ containing a closed $G$-orbit of
maximal (in this fibre) dimension, coincides with this orbit.

\begin{prop}\label{lem.catq-fibration}
 Let $X$ be a geometrically integral affine
$k$-variety with an action of a reductive $k$-group $G$ such that
$X(k)$ is Zariski dense in $X$. Let $\pi\colon X\to X \catq G$ be a
categorical quotient. Then the following properties are equivalent:
\begin{enumerate}
 \item[\rm(a)] the action of $\;G$ on $X$ is stable;
     \item[\rm(b)] there exist  a reductive
 $k$-subgroup $\,S$ of $\;G$ and  a
dense open subset $\;Y$ of  $\;X \catq G$ such that the restriction
of $\;\pi$ to $\pi^{-1}(Y)$ is a $(G,S)$-fibration  $\pi^{-1}(Y) \to
Y$.
\end{enumerate}

The group $S$ in {\rm (b)} may be taken as the $G$-stabilizer of any
$k$-point of $\;\pi^{-1}(Y)$.

If {\rm(a), (b)} hold, then $\pi$ induces an isomorphism
$\pi^*\colon k(X\catq G)\overset{\cong}{\to} k(X)^G$.
\end{prop}

\begin{proof}

Assume that (a) holds. By Theorem \ref{richardson1}, there exist a
reductive $k$-subgroup $S$ of $G$ and a $G$-invariant open subset
$U_1$ of $X$ such that $U_1$ admits the structure of a $(G,
S)$-fibration $\alpha \colon U_1 \to Z_1$.

On the other hand, by (a) there is an open subset $U_2$ of $X$ such
that the $G$-orbit of every point of $\;U_2(\overline k)$ is closed.
Since there is an open subset $U_{\rm max}$ of $X$ such that the
$G$-orbit of every point of $U_{\rm max}(\overline k)$ has maximal
(in $X$) dimension (cf.\;\cite[Chap.\,0, \S2]{mumford} or
\cite[Sect.\;1.4]{popov-vinberg}), we may replace $U_2$ by $U_2\cap
U_{\rm max}$ and assume in addition that this maximality property
holds for every point of $U_2(\overline k)$. The openness of $U_2$
in $X$ implies that  $\pi(U_2)$ contains a smooth open subset $Y_1$
of $X \catq G$. Put $U_3:=\pi^{-1}(Y_1)$. Then the fibre of  $\pi$
over every point of $Y_1(\overline k)$ contains a closed $G$-orbit
of maximal dimension. As we mentioned  right before the statement of
Proposition \ref{lem.catq-fibration}, this implies that this fibre
is a $G$-orbit. In turn, as we mentioned in Section
\ref{sect.quotients}, this implies that $\pi|_{U_3}\colon U_3\to
Y_1$ is the geometric quotient for the $G$-action on $U_3$.

Let $U = U_1 \cap U_3$. Then, since $\alpha$ and $\pi|_{U_3}$ are
open morphisms, $Z:=\alpha(U)$ and $Y:=\pi(U_3)$ are open subsets of
$Z_1$ and $Y_1$ respectively. The morphisms $\pi|^{\ }_U \colon U
\to Y$ and $\alpha|^{\ }_U \colon U \to Z$ are geometric quotient
maps for the $G$-action on $U$. By uniqueness of geometric
quotients, there is an isomorphism $\varphi\colon Z\to Y$ such that
the diagram
\[ \xymatrix@R=4mm@C=10mm{ & U \ar@{->}[dl]_{\alpha|^{\ }_U} \ar@{->}[dr]^{\pi|^{\ }_U} & \cr
Z \ar@{->}[rr]_{\varphi}^{\cong} & & Y} \] is commutative. Hence (b)
holds.

Conversely, if (b) holds, then fibres of $\pi $ over points of
$Y(\overline k)$ are $G$-orbits. Therefore these orbits are closed;
whence (a).

To prove the last assertion of the proposition, we may replace $X
\catq G$ by $Y$ and thus assume that $\pi$ is a $(G, S)$-fibration.
By Proposition~\ref{properties}(iv) $\pi$ is a geometric quotient
map; the desired conclusion now follows from
Remark~\ref{rem.separate}.
\end{proof}

%%%%%%%%%%%%%%%%%%%%%%%%
%%%%%%%%%%%%%%%%%%%%%%%%
%%%%%%%%%%%%%%%%%%%%%%%%

\section{Versal actions}
\label{sect.versal}

We recall that $k$ is a field of characteristic zero. Let $\kbar$ be
an algebraic closure of $k$,  and let $G$ be a (not necessarily
connected) linear algebraic group over $k$.

\begin{defn} \label{def.versal}
Let $S$ be a closed $k$-subgroup of $G$. We say that a
$(G,S)$-fibration $\pi \colon V \to Y$ is {\em versal} if $Y$ is
geometrically integral and for every field extension $L/k$, every
$(G_{L},S_{L})$-fibration $\varrho \colon X \to \Spec(L)$, and every
dense open subset $Y_0$ of $Y$, there exists a Cartesian diagram of
the form\ \vskip -4mm
\begin{equation} \label{e.def3-1}
\begin{matrix}
\xymatrix@C=6mm{ X \ar@{->}[d]_{\varrho} \ar@{->}[rr] & & V
\ar@{->}[d]^{\pi} \cr \Spec(L) \ar@{->}[r] & Y_0 \ar@{^{(}->}[r] & Y
}
\end{matrix}\quad.
\end{equation}
\vskip 2mm\noindent In other words, there is an $L$-point of $Y_{0}$
and an $L$-isomorphism between $X$ and the fibre product $V
\times_{Y }\Spec(L) $.
\end{defn}

Note that if $S = \{ 1 \}$, i.e., $\pi$ is a $G$-torsor, this
definition coincides with the usual definition of a versal torsor;
see~\cite[Sect.\,I.5]{gms}, \cite{berhuy-favi}.

\begin{lem} \label{lem.versalG/S}
Let $N$ be the normalizer of $\;S$ in $G$ and let $H=N/S$. A $(G,
S)$-fibration $\pi \colon V \to Y$ over a smooth $Y$ is versal if
and only if the associated $H$-torsor $\pi^S:=\pi|_{V^S} \colon V^S
\to Y$ is versal.
\end{lem}

\begin{proof} By Proposition~\ref{propGSfibration} there are
mutually inverse functorial correspondences between $(G,
S)$-fibrations and $H$-torsors over $\Spec(L)$ given by passing from
a $(G, S)$-fibration $\varrho\colon X \to \Spec(L)$ to the
$H$-torsor $\varrho^S\colon X^S\to \Spec(L)$ and from an $H$-torsor
$\alpha\colon Z\to \Spec(L)$ to the $(G, S)$-fibration $\alpha^{\
}_{G/S}\colon (G/S)\times^H Z\to  \Spec(L)$. This implies that a
Cartesian diagram~\eqref{e.def3-1} exists if and only if a Cartesian
diagram \ \vskip -4mm
\begin{equation*}
\begin{matrix}
\xymatrix@C=6mm{ X^S \ar@{->}[d]_{\varrho^S} \ar@{->}[rr] & & V^S
\ar@{->}[d]^{\pi^{S}} \cr \Spec(L) \ar@{->}[r] & Y_0 \ar@{^{(}->}[r]
& Y }
\end{matrix}\quad .
\end{equation*}
\vskip 2mm\noindent exists. This means that $\pi$ is versal if and
only if $\pi^S$ is versal.
\end{proof}

We say that a $(G, S)$-variety $X$ is {\em versal} if there is a
friendly open set $U$ of $X$ (see Definition~\ref{defGSvariety1})
such that the associated $(G, S)$-fibration $U \to Y$ is versal.

It is easy to see that if $X$ is a versal $(G, S)$-variety, then the
$(G, S)$-fibration $U' \to Y'$ is a  versal $(G, S)$-fibration for
{\em every} friendly open set $U'$ of $X$.

The following proposition plays an important r\^ole in our paper. We
are grateful to the referee for the present version of this
statement which strengthens our earlier result used in the proof of
Theorem \ref{thm1}. Recall that the notion of stabilizer in general
position utilized in the formulation of this proposition has been
defined in the previous section, just after the proof of
Theorem~\ref{generalGS}.

\begin{prop} \label{prop.versal0}
Suppose $G$ and $H$ are linear algebraic groups over $k$ and $G$
acts {\rm(}algebraically\,{\rm)} on $\,H$ by group automorphisms.
Assume further that
\begin{enumerate}
\item[\rm(i)] the group
 $\,H$ is connected and,
  for this action of $\,G$ on $H$, there exists a stabilizer $\,S$
in general position;

\item[\rm(ii)] the group $\,H^S$ is connected.

\end{enumerate}

\noindent Then $H$ is a versal $(G, S)$-variety.
\end{prop}

Note that by Theorem~\ref{richardson1}(i) condition (i)
automatically holds if $G$ is reductive.

Our proof of Proposition~\ref{prop.versal0} will rely on the
following two lemmas.

\begin{lem} \label{lem.twist}
Let $C$ be an algebraic group over a field $K$, let $X$  be a
quasi-projective $K$-variety endowed with a $C$-action, and let $\pi
\colon P \to \Spec(K)$ be a $C$-torsor. Then the following
properties are equivalent:

\begin{enumerate}
\item[\rm(a)]
There exists a $C$-equivariant morphism $\alpha\colon P \to X$
defined over $K$.

\item[\rm(b)] $X
\times^ C \!P$ has a $K$-point.
\end{enumerate}
\end{lem}

Recall that here $X \times^C \!P$ is the $K$-variety defined after
the proof of Proposition~\ref{GSfibration0}.

\begin{proof}[Proof of Lemma~{\rm\ref{lem.twist}}]
${\rm(a)}\Rightarrow\!{\rm(b)}$: By the universal mapping property
(see Subsection~\ref{sect.quotients}) the $C$-equivariant morphism
$\alpha \times \id \colon P \to X \times P$ determines
 a $K$-morphism of geometric quotients
$\Spec(K) = P/C \to ( X \times P)/C =X \times^ C \!P$, i.e., a
$K$-point of $X \times^ C \!P$.

\smallskip
${\rm(b)}\Rightarrow\!{\rm(a)}$: Given a $K$-point $\mu\colon
\Spec(K) \to X \times^C\! P$, let $\varepsilon\colon E\to \Spec(K)$
be the $C$-torsor obtained from the $C$-torsor $X\times P\to X
\times^C\! P$ by the base change $\mu$. By construction there is a
$C$-equivariant  morphism $E\to X\times P$ (see
Subsection~\ref{GSfibrations}). Its composition with the projection
$X\times P\to P$ is a morphism of $C$-torsors $\pi$ and
$\varepsilon$, hence an isomorphism. This yields a $C$-equivariant
morphism $P\to X\times P$. Its composition with the projection
$X\times P\to X$ is a $C$-equivariant morphism $P\to X$.
\end{proof}

\begin{lem} \label{lem.group}
Let $C$, $X$, and $P$ be as in Lemma~{\rm\ref{lem.twist}}. If $X$ is
an algebraic group and $C$ acts on $X$ by group automorphisms, then
$X\times^C\!P$ has a natural structure of an algebraic group defined
over $K$.
\end{lem}

\begin{proof}[Proof of Lemma~{\rm\ref{lem.group}}]
For notational simplicity we shall write $^P \!X$ in place of $X
\times^C \!P$.

Let $X_1$ and $X_2$ be quasi-projective $K$-varieties
 endowed with
$C$-actions  and let $ \varphi\colon X_1\to X_2$ be a
$C$-equivariant morphism. By the universal mapping property the
$C$-equivariant morphism $ \varphi \times \id \colon X_1 \times P
\to X_2 \times P $ determines a morphism of geometric quotients $^P
\!X_1 \to {^P\!X}_2$ which we shall denote by $^P \!\!\varphi$.

For $i = 1, 2$, let $\pi_i \colon X_1 \times X_2 \to X_i$  be the
projection. We claim that the $K$-morphism
\begin{equation} \label{e.twist}
^P \!\!\pi_1 \times {^P \!\!\pi}_2 \colon ^P\!(X_1 \times X_2) \to
{^P\!X}_1 \times {^P\!X}_2
\end{equation}
is, in fact, an isomorphism. To prove this claim we pass to a finite
field extension $K'$ of $K$ such that $P$
 splits over $K'$,
next we observe that if $P$ splits, the claim is obvious. We
conclude that~\eqref{e.twist} is an isomorphism over $K'$ and hence
by \cite[Vol.\,24, Prop.\,2.7.1(viii)]{EGA4} over $K$.

Using this isomorphism one easily checks that the multiplication map
$X \times X \to X$ and the inverse map $X \to X$ give rise to group
operations on $^P \!X$.
\end{proof}

\begin{proof}[Proof of Proposition~{\rm\ref{prop.versal0}}]
Condition (i) and Theorem \ref{generalGS} yield that, for the
$G$-action on $H$, there is
 a friendly open subset $U$ of $H$. So a geometric factor
$U \to Y$  for the action of $G$ on $U$ exists and is a $(G,
S)$-fibration. The action of $N_G(S)/S$ on $H^S$ is generically free
and Proposition~\ref{GSfibration0}(i) yields that $U^S = H^S \cap U$
is a friendly open subset of  $H^S$ for this action. By
Lemma~\ref{lem.versalG/S} it suffices to show that $U^S \to Y$ is a
versal $N_G(S)/S$-torsor or, equivalently, that $H^S$ is a versal
$N_G(S)/S$-variety. Thus, given condition (ii), after replacing $G$
by $N_G(S)/S$ and $H$ by $H^S$, we may assume that $S = \{ 1 \}$.

Our goal now is to show that, under this assumption, the $G$-action
on the friendly open subset $U$
 of $H$ satisfies the conditions
of Definition~\ref{def.versal}, i.e.,
 for every field extension $L/k$,
every $G$-torsor $ P \to \Spec(L)$, and  every open dense subset
$Y_0$ of $Y$, there is a $G$-equivariant map $P \to U_0 :=
\pi^{-1}(Y_0)$. By Lemma~\ref{lem.twist} with $C = G_L$ this map
exists if and only if $(U_0)_L \times^{G_L}\!P$ has an $L$-point.
Since $U_0$ is a dense $G$-invariant subset of $U$ (and hence of
$H$), we see that $(U_0)_L \times^{G_L}\!P$ is a dense open subset
of $H_L \times^{G_L}\!P$.  Thus, it suffices to show that $L$-points
are dense in $H_L \times^{G_L}\!P$.

Lemma~\ref{lem.group} implies that $H_L \times^{G_L} P$ is an
algebraic group over $L$. It is connected because $H$ is connected.
Let $L'$ be a finite field extension of $L$ such that $P$ splits
over $L'$. Then the $L'$-groups  $(H_L \times^{G_L} P)\times_{L} L'$
and $H_L\times _{L} L'$  are isomorphic. Since $H$ is linear, this
yields that $(H_L \times^{G_L} P)\times_{L} L'$ is linear. The
standard descent result \cite[Vol.\,24, Prop.\,2.7.1(xiii)]{EGA4}
then implies that the $L$-group $H_L \times^{G_L} P$ is linear. But
since ${\rm char}(L)=0$, by Chevalley's
theorem~\cite[Theorem~18.2(ii)]{borel} any connected linear
algebraic group defined over $L$ is unirational over $L$. Hence $H_L
\times^{G_L} P$  is unirational over $L$ and therefore $L$-points
are dense in it, as claimed.
\end{proof}

\begin{cor} \label{cor.versal-c}\
\begin{enumerate}
\item[\rm(a)] {\rm(}cf.\;{\rm\cite[Prop.\;7.1]{reichstein}} and {\rm\cite[Ex.\,I.5.4]{gms})}
Every finite-dimensional generically free $G$-mo\-du\-le $V$ defined
over $k$ is a versal $(G,\{1\})$-va\-ri\-ety.

\item[\rm(b)] If $\;G$ is reductive and $V$ is a
finite-dimensional $G$-module defined over $k$, then $V$ is a versal
$(G, S)$-variety for a suitable $k$-subgroup $S$ of $\;G$.  There
exists a dense
 open subset $U$ of $\;V$ such that the $G$-stabilizer of each
point of $\;U(k)$ is a possible choice for $S$.
\end{enumerate}
\end{cor}

\begin{proof}
In both parts view $V$ as the unipotent $k$-group $\,{\bf G}_a^{\dim
V}$ and apply Proposition~\ref{prop.versal0} and, in part (b),
Theorem~\ref{richardson1}(i).
\end{proof}

\begin{lem} \label{lem.fibreing}
Let $\varphi \colon X \to Y$ be a dominant morphism of integral
$k$-varieties.  Denote the generic point of $Y$ by $\eta$ and the
generic fibre of $\varphi $ by $X_{\eta}$.

\begin{enumerate}
 \item[\rm(a)]
Suppose $\varphi $ has a rational section $s \colon Y \dasharrow X$.
Then there exists a dense open subset $Y_0$ of $\;Y$ defined over
$k$ such that for any morphism $Z \to Y_0$ of integral schemes, the
natural projection $\varphi_Z\colon X_Z:=X\times_Y Z \to Z$ has a
section $Z \to X_Z$.

 \item[\rm(b)] Suppose the generic fibre $X_{\eta}$ is connected and
is rational {\rm(}respectively, stably rational\,{\rm)} over $k(Y)$.
Then there exists a dense open subset $Y_0$ of $\;Y$ defined over
$k$ such that for any extension field $L/k$ and any point $y\in
Y_{0}(L)$, the fibre $X_{y}$ of $\;\varphi$ over $y$ is integral and
rational {\rm(}respectively, stably rational\,{\rm)} over $L$.
\end{enumerate}
 \end{lem}

\begin{proof} (a) Choose $Y_0 \subset Y$ so that $s$ is regular on $Y_0$
and pull back the section $s$ to $X_Z$.

\smallskip
Before we prove (b), let us  discuss a more general situation. Let
$p: X \to Y$ and $p' : X' \to Y$ be two dominant $k$-morphisms of
geometrically integral $k$-varieties with geometrically integral
generic fibres.
 Assume that the generic fibres are birationally isomorphic
 over $k(Y)$. Then $X$ and $X'$ are birationally isomorphic over $Y$. There thus exist two dense
open sets $U \subset X$ and $U' \subset X'$ and a $Y$-isomorphism $U
\overset{\cong}\to U'$.  Let $Y_{0}\subset p(U)$ be a Zariski dense
open set and replace $U$ and $U'$ by their restrictions over
$Y_{0}$. Then all geometric fibres of $U \to Y_{0}$ and $U' \to
Y_{0}$ are nonempty. For any point $y \in Y_{0}$ this induces a
$k(y)$-isomorphism between the nonempty fibre $U_{y} \subset X_{y}$
and the nonempty fibre $U'_{y} \subset X'_{y}$. The same therefore
holds over $L$ with $k(y) \subset L$.

To prove statement (b) in the case where $X_{\eta}$ is rational,
 it suffices   to apply this argument to
$X$ and $X'={\bf P}^d_{Y}$ where $d$ is the dimension of $X_{\eta}$
and ${\bf P}^d_{Y}$ is the $d$-dimensional projective space over
$Y$.

To prove statement (b) in the case where $X_{\eta}$ is stably
rational,
 it suffices   to apply this argument to  the pair
$X \times_{Y}{\bf P}^n_{Y}$ and $ {\bf P}^{n+d}_{Y}$, with $d$ as
above and $n$ some positive integer.
\end{proof}

\begin{defn}
Given a $(G,S)$-variety $X$, we shall say that it admits a rational
section if for some and hence any friendly open set $U \subset X$
the quotient map $U \to U/G$ admits a section over a dense open set
of $U/G$.
\end{defn}

\begin{thm} \label{thm1'} Let $G$ be a   linear algebraic group over $k$, let
$S$ be a closed $k$-subgroup of $\;G$, and let $V$ be a
geometrically integral ver\-sal $(G, S)$-variety.

\begin{enumerate}
 \item[\rm (a)] If the $(G,S)$-variety $V$ admits
 a rational section,  then for every field extension $F/k$
  every
  geometrically integral
  $(G, S)$-variety $X$ over $F$
admits a rational section.

 \item[\rm(b)] Assume that the homogeneous space $G/S$ is connected.
If $k(V)/k(V)^G$ is pure {\rm(}respectively, stably pure{\rm)},
then  for every field extension $F/k$ and every geometrically
integral $(G, S)$-variety $X$ over $F$,  the field extension
$F(X)/F(X)^G$ is pure {\rm(}respectively, stably pure{\rm)}.
\end{enumerate}
\end{thm}

\begin{proof}
 Note that  if $V$  is
a versal $(G, S)$-variety over $k$ then $V_F$ is a versal $(G,
S)$-variety over $F$. Since the hypothesis $F(V)/F(V)^G$ pure, or
stably pure, holds as soon as it does for $k(V)/k(V)^G$, it is
enough to prove the theorem for $F=k$. After replacing $V$ by a
friendly open subset we may assume that we are given a
$(G,S)$-fibration $\pi \colon V \to Y$. Choose a dense open subset
$Y_0 \subset Y$ as in Lemma~\ref{lem.fibreing}(a).

(a) After replacing $X$ by a friendly open subset, we may assume
that $X$ is the total space of a $(G,S)$-fibration $\alpha \colon X
\to Z$. Let $\eta$ be the generic point of $Z$. Since $\pi$ is
versal, the $(G,S)$-fibration $\alpha_{\eta} \colon X_{\eta} \to
\eta$ can be obtained  by pull-back from $\pi$ via a morphism $\Spec
\, k(Z) \to Y_0$. In other words, after replacing $X$ by a smaller
friendly open set, we may assume that $\alpha \colon X \to Z$ is the
pull-back of $\pi \colon V \to Y$ via a morphism $Z \to Y_0 \subset
Y$. The desired conclusion now follows from
Lemma~\ref{lem.fibreing}(a).

The proof of part (b) is exactly the same, except that we appeal to
Lemma~\ref{lem.fibreing}(b), rather than to
Lemma~\ref{lem.fibreing}(a).
\end{proof}

\begin{lem} \label{pointunirat}
Let $G$  be a connected linear algebraic  group over $k$ and let $X$
be a geometrically integral $k$-variety with  $G$-action which
admits  a geometric quotient $\pi\colon X  \to Y$. The following
properties are equivalent:
\begin{enumerate}
 \item[\rm(a)]
 $\pi \colon X \to Y$  admits a rational section;

 \item[\rm(b)]
 $k(X)$ is unirational over $k(X)^G$.
     \end{enumerate}
\end{lem}
\begin{proof}
We know that $\pi$ induces an isomorphism $\pi^* \colon k(Y) \oii
k(X)^G$. Let $X_{\eta}$ be the $k(Y)$-variety which is the generic
fibre of $\pi \colon X \to Y$.

Assume (b). The hypothesis   implies that there exists a dominant
$k(Y)$-rational map $\varphi$ from some projective space
$\bbP^n_{k(Y)}  $ to $X_{\eta}$. This rational map is defined on a
dense open set $U \subset \bbP^n_{k(Y)}$. Since rational points are
Zariski dense on projective space over an infinite field, the set
$U(k(Y))$ is nonempty. The $k(Y)$-morphism $\varphi : U \to
X_{\eta}$ sends such a point to a $k(Y)$-point of $X_{\eta}$, i.e.,
to a rational section of $X \to Y$. Thus (a) holds.

  Assume (a) holds.
By the definition of $\pi$, the generic fibre of $\pi$ is a
$k(Y)$-variety with function field $k(X)$, and it is a homogeneous
space of $G_{k(Y)}$ which  by (a)  admits a $k(Y)$-rational point.
We thus have inclusions of fields $k(Y) \subset k(X) \subset
k(Y)(G)$. By a theorem of Chevalley, over a field of characteristic
zero, any connected linear algebraic group is unirational (see
\cite[Theorem\;18.2]{borel} or \cite[Vol.\;II, Chap.\;XIV,
Cor.\;6.10]{sga3}). Thus $k(Y)(G)$ embeds into a purely
transcendental extension of $k(Y) \oii k(X)^G$. This proves (b).
\end{proof}
%%%%%%%%%%%%%%%%%%%%%%%%
%%%%%%%%%%%%%%%%%%%%%%%%
%%%%%%%%%%%%%%%%%%%%%%%%

\section{The conjugation action and the adjoint action}
\label{conjugationandadjointaction}

We now concentrate on the main actors. We recall that $k$ is a field
of characteristic zero. Let $G$ be a connected reductive group over
$k$.

\subsection{Quotients by the adjoint action,
versal \boldmath $(G,S)$-varieties, and Kostant's theorem}

  The radical ${\Rad}(G)$ of $\,G$ is
a  central $k$-torus in $G$. The $G$-stabilizer of a point $g\in G$
for the conjugation action of $\,G$ on itself is the centralizer of
$g$ in $G$. There is a Zariski dense open set of $G$ such that the
centralizers of its points in $G$ are maximal tori of $G$, see
\cite[12.2, 13.1, 13.17, 12.3]{borel}.

\begin{lem} \label{closedness} The following properties of an element $g\in G$ are equivalent:
\begin{enumerate}
\item[\rm(i)]
the conjugacy class of $\,g$ is closed in $G$;
 \item[\rm(ii)] $g$ is semisimple.
\end{enumerate}
\end{lem}
\begin{proof} For  semisimple groups
this is proved in \cite[6.13]{steinberg1}. The general case can be
reduced to that of semisimple groups in the following manner. Let
$(G, G)$ be the commutator subgroup of $G$. It is a closed,
connected, semisimple $k$-subgroup of $G$, and $G=(G, G)\cdot {\rm
Rad}(G)$, see \cite[2.3, 14.2]{borel}. Let $g=hz$ for some $h\in (G,
G)$, $z\in {\rm Rad}(G)$, and let $g=g_sg_u$, $h=h_sh_u$ be the
Jordan decompositions, see \cite[4.2]{borel}. Since ${\rm Rad}(G)$
is a central torus, $g_s=h_sz$, $g_u=h_u$. Hence $g$ is semisimple
if and only if $h$ shares this property. Let $G\cdot g$ and $G\cdot
h$ be respectively the $G$-conjugacy classes of $g$ and $h$ in $G$.
Since $z$ is central, $G\cdot g=(G\cdot h)z$. Hence closedness of
$G\cdot g$ in $G$ is equivalent to that of $G\cdot h$. But $G\cdot
h$ coincides with the $(G, G)$-conjugacy class of $h$ in $(G, G)$
and, since $(G, G)$ is semisimple, the cited result in
\cite{steinberg1} shows that the latter is closed in $(G, G)$ if and
only if $h$ is semisimple. This completes the proof.
\end{proof}

\begin{cor}\label{stability_conj} The action of $\,G$ on itself by conjugation is stable.
\end{cor}

Analogous statements hold for the adjoint  action of $G$ on $\g$,
see~\cite{kostant}.

\begin{prop} \label{prop.versal}
Let $G$ be a connected reductive group  over $k$ and $\g$ its Lie
algebra. Let $S \subset G$ be a maximal $k$-torus. Let $X$ be either
$G$ or $\g$ and let $\pi \colon X \to  Y :=X \catq G$ be the
categorical quotient for the conjugation, respectively, the adjoint
action.  Then
\begin{enumerate}
 \item[\rm(a)] there exists a dense Zariski open subset  $V$ of $\;Y$
with inverse image $U=\pi^{-1}(V)  $ such that  $\pi|_U \colon U \to
V$ is a $(G,S)$-fibration;

\item[\rm(b)] $\pi^* $ induces an isomorphism $k(Y) \oii k(X)^G$;

\item[\rm(c)] $X$ is a versal $(G, S)$-variety.
     \end{enumerate}
\end{prop}

\begin{proof}
Statements (a) and (b) follow from Corollary \ref{stability_conj}
and Proposition \ref{lem.catq-fibration} by virtue of the
identification of general $G$-stabilizers with maximal tori of  $G$,
which are all conjugate over $\kbar$.

For  $X = \g$, part (c) is a special case of
Corollary~\ref{cor.versal-c}(b).

To prove part (c) for $X=G$, note that
\[ \text{$X^S$ = centralizer of $S$ in $G$ = $S$} \]
is connected. Hence Proposition~\ref{prop.versal0} applies to the
conjugation action of $G$ on itself, yielding the desired
conclusion.
\end{proof}

The following well known result plays an important r\^ole in the
sequel.

\begin{prop} [{\rm Kostant}]  \label{prop.kostant}
Let $G$ be a reductive linear algebraic group over $k$  and $\g$ be
its Lie algebra. Assume that the semisimple quotient $G/{\Rad}(G)$
is quasisplit. Then the categorical quotient map $\pi \colon \g \to
\g \catq G$ has a {\rm(}regular{\rm)} section.
\end{prop}

\begin{proof}
For algebraically closed base field $k$ this is a theorem of Kostant
 \cite[Theorem 0.6]{kostant}.
For an arbitrary base field $k$ of characteristic $0$, see~\cite[\S
4.3]{kottwitz2}.  \end{proof}

\begin{prop} \label{cor.steinberg1}
Let $G $ be a   connected reductive group over $k$, let $S$ be a
maximal $k$-torus of $G$, and let $X$ be a geometrically integral
$(G, S)$-variety over $k$. Assume that the semisimple group
$G/{\Rad}(G)$ is quasisplit.  Then
\begin{enumerate}
 \item[\rm(a)]
 $X$ admits a rational section;

 \item[\rm (b)]
 $k(X)$ is unirational over $k(X)^G$.
     \end{enumerate}
\end{prop}

\begin{proof}
 Going over to a friendly open set, we may assume that
we are given a $(G,S)$-fibration $\pi \colon X \to Y$. Since the
radical of $G$ lies in every conjugate of $S$, it acts trivially on
$X$.  Thus the $G$-action on $X$ descends to the semisimple group
$G/{\Rad}(G)$, and we may assume that the $k$-group $G$ is
quasisplit semisimple.

 By Proposition~\ref{prop.versal}(c), the Lie algebra
 $\g$ equipped with the adjoint action is
 a versal $(G, S)$-variety.
By Proposition~\ref{prop.kostant} the map $\pi \colon \g \to \g
\catq G$ admits a section. Statement (a) now follows from
Theorem~\ref{thm1'}(a). As for (b), by Lemma~\ref{pointunirat} it
follows from (a).
\end{proof}

Many instances of the following immediate corollary have appeared in
the literature (cf.~\cite{kottwitz}).

\begin{cor} \label{cor.hom-space}
Let $G$ be a connected reductive group over a field $k$, let $K$ be
an overfield of $k$. Let   $X$ be a $K$-variety which is a
homogeneous space of $\;G_{K}$. If the semisimple quotient
$\;G/{\Rad}(G)$ is quasisplit and if the geometric stabilizers of
the $G_{K}$-action on $X$ are maximal tori in $G_{\overline{K}}$,
then $X$ has a $K$-rational point. \qed
\end{cor}

\begin{remark} \label{rem.hom-space}
Suppose that the base field $k$ is algebraically closed and $G$ is a
connected simple algebraic group defined over $k$. Let $V$ be a
faithful simple $G$-module over $k$. Theorem \ref{richardson1} tells
us that $V$ is a $(G, H)$-variety for some closed subgroup $H$ of
$G$. The list of all pairs $(G, H)$, with $H \ne \{ 1 \}$, which can
occur in this setting can be found
in~\cite[pp.\,260--262]{popov-vinberg}. Then
 the analogue of Corollary~\ref{cor.hom-space} holds,
namely every $G$-homogeneous space $X$, defined over $K$ whose
geometric stabilizers are isomorphic
 to $H\times_{k}{\overline K}$ has a $K$-point.
The proof is similar to the one above, except that instead of
Kostant's result (Proposition~\ref{prop.kostant}), one uses the
existence of a regular section for the categorical quotient map $V
\to V \catq G$, proved in~\cite{popov2}.
\end{remark}

Part (a) of the following corollary partially generalizes a result
of Steinberg~\cite{steinberg1}, who constructed a regular section of
$\pi$ in the case where $G$ is simply connected.

\begin{cor} \label{cor.steinberg2} Let $G$ be a connected reductive
group over $k$ and let $\g$ be its Lie algebra. If the semisimple
quotient $\;G/{\Rad}(G)$ is quasisplit, then
\begin{enumerate}
 \item[\rm(a)] the categorical  quotient map $\pi \colon G \to G \catq  G$
for the conjugation action of $G$ on itself admits a rational
section;

 \item[\rm(b)] $k(G)$ is unirational over $k(G)^G$;

 \item[\rm (c)] $k(\g)$ is unirational over $k(\g)^G$.
     \end{enumerate}
\end{cor}

\begin{proof}
By Proposition~\ref{prop.versal} both $G$ and $\g$ are
$(G,S)$-varieties. All three parts now follow from
Proposition~\ref{cor.steinberg1}.
\end{proof}

\begin{remark} \label{rem.cayley} Recall from~\cite{lpr} that
a {\em Cayley map} for an algebraic group $G$ over a field $k$ is a
$G$-equivariant birational isomorphism
\begin{equation} \label{e.Cayley-map}
\varphi \colon G \dasharrow \g,
\end{equation}
where $\g$ is the Lie algebra of $G$. Here, as before, $G$ is
assumed to act on itself by conjugation and on $\g$ by the adjoint
action. We say that $G$ is a {\em Cayley group} if it admits a
Cayley map.  In particular, the special orthogonal group $\SO_n$,
the symplectic group $\Sympl_{2n}$, and the projective linear group
$\PGL_n$ are Cayley groups for every $n \ge 1$ and every base field
$k$ of characteristic zero; see~\cite[Examples 1.11 and 1.16]{lpr}.

In the case where $G$ is a Cayley group,
Theorem~\ref{thm.steinberg0} (or equivalently,
Corollary~\ref{cor.steinberg2}(a)) has the following simpler proof.
By Proposition~\ref{prop.kostant} the categorical quotient map
$\pi_{\g}\colon \mathfrak g\to \g \catq G$ has a section $\sigma
\colon \g \catq G \to \g$. Let $\pi^{}_G\colon G\to G/\!\!/G$ be the
categorical quotient map. Then the Cayley map~\eqref{e.Cayley-map}
induces a commutative diagram
\[ \xymatrix{
G \ar@{->}[d]_{\pi^{}_G} \ar@{-->}[r]^\varphi & \g
\ar@{->}[d]_{\pi^{}_{\g}} \cr G \catq G
\ar@{-->}[r]^{\varphi/\!\!/G} & \g \catq G
\ar@{->}@/_1pc/[u]_\sigma,}
\]
where $\varphi$ and $\varphi/\!\!/G$ are birational isomorphisms.
The section $\sigma$ pulls back to a rational section of $\pi^{}_G$
via this diagram.

In the case of $G = \PGL_2$ explicitly mentioned by Grothendieck
(see the footnote in the introduction), we have ${\mathfrak
g}={\mathfrak s}{\mathfrak l}_2$ and
\begin{enumerate}
\item[(i)] $G \catq G$ is
the affine line $\bbA^1$ and $ \pi^{}_{\PGL_2}\hskip -1mm
\colon\hskip -1mm \PGL_2 \to \bbA^1$ is given by $ [g]\mapsto (\tr
g)^2/ \det g$, where $\GL_2\to\PGL_2$,
$g=\bigl(\begin{smallmatrix}a&b\\c&d
\end{smallmatrix}\bigr)
\mapsto[g]:=\bigl[\begin{smallmatrix}a&b\\c&d
\end{smallmatrix}\bigr]$, is the natural projection;
\item[(ii)]  ${\mathfrak g}/\!\!/G$ is the affine
line ${\mathbb A}^1$ and $\pi^{}_{\mathfrak g}\colon {\mathfrak
g}\to {\mathbb A}^1$ is given by $g\mapsto {\rm det}\,g$.
\end{enumerate}

In this case a Cayley map $\varphi\colon {\rm PGL}_2\dasharrow
{\mathfrak s}{\mathfrak l}_2$
 is given by
\[
[g]\mapsto \frac {2}{\tr g} g  - I_2, \] where $I_2$ is the $2\times
2$ identity matrix (see \cite[Example 1.11]{lpr}), the map
$\varphi/\!\!/G\colon {\mathbb A}^1\dasharrow {\mathbb A}^1$ is
given by $t\mapsto -1+4/t $, and the section $\sigma\colon {\mathbb
A}^1\to {\mathfrak s}{\mathfrak l}_2$ is given by $t\mapsto
\bigl(\begin{smallmatrix}0&1\\
-t&0\end{smallmatrix}\bigr)$. The above strategy then leads to the
rational section
\[\bbA^1
\dasharrow \PGL_2,\qquad t \mapsto
\begin{bmatrix}  1  &   1 \\
                         1 - 4/t &   1
\end{bmatrix}.
\]
\end{remark}

\subsection{The generic torus}
\label{thegenerictorus}

The conjugation action of a connected reductive group $G$ on itself
leads to another construction, that of the generic torus.
 Let $S$ be a maximal $k$-torus of the connected reductive $k$-group $G$ and
let $N$ be the normalizer of $S$ in $G$.
 Consider the natural map
$\varphi \colon S \times_k (G/S)  \to G \times_k (G/N) $ given by
$(s, gS) \mapsto (gsg^{-1}, gN)$. Its image $H \subset G \times_k
(G/N)$ is closed (see \cite[p.\,10]{humphreysconj}). The point
$\varphi(s,g)$ defines the maximal torus $gSg^{-1}$ and the point
$gsg^{-1}$ in that torus. The second projection $\pi\colon H \to
G/N$ gives $H$ the structure of a torus over $G/N$, and the family
of fibres of this projection
  is the family of maximal tori in $G$.
To be more precise, the morphism  $H \hookrightarrow G \times_k
(G/N)$ is a morphism of $G/N$-group schemes, where $H$ is a
$(G/N)$-torus and $ G \times_{k} (G/N)$ is the constant
$(G/N)$-group scheme induced by base change from $G \to \Spec k$.
The $(G/N)$-torus $H$ is thus a maximal torus in the (fibrewise
connected) reductive $(G/N)$-group $G \times_k G/N$. The variety
$G/N$ is the ``variety of maximal tori in $G$". Given any field
extension $L/k$ and any maximal $L$-torus $S$ in $G_{L}$, there
exists an $L$-point $s \in (G/N)(L)$ such that $\pi^{-1}(s)=S$. The
actions of $G$ by conjugation on itself and by left translation on
$G/N$ induce a $G$-action on $G\times_k G/N$ with respect to which
$H$ is stable and $\pi$ is $G$-equivariant. Since $G(L)$ is dense in
$G$, this implies that the set of $L$-points of $G/N$ whose fibre
under $H \to G/N$ is isomorphic to a given $L$-torus is Zariski
dense in $G/N$. The field $k(G/N)$ is {\it denoted} $K_{\rm gen}$.
The {\it generic torus} $T_{\rm gen}$ is by definition the generic
fibre of $\pi$. For the details of this
construction,~see~\cite[\S4.1--4.2]{voskresenskii}.

Assume that $G$ is split over $k$ and $S$ is a split maximal torus
of $G$. Then the $K_{\rm gen}$-torus $T_{\rm gen}$ is split by the
extension $k(G/S)$ of $k(G/N)$, which is a Galois extension with
Galois group the Weyl group $W=N/S$. If,  moreover, $G$ is simple,
simply connected of type ${\sf R}$, then the character lattice of
the generic torus is the weight $W$-lattice $P({\sf R})$. For proofs
of these assertions, see~\cite{voskresenskii2}.

\vskip -7mm

\subsection{Equivalent versions of the purity questions}
We consider the purity Questions~\ref{q1}(a) and (b) from the
Introduction.

\begin{thm} \label{thm.reduction}
Let $G$ be a connected reductive  group over $k$ and let $S$ be a
maximal $k$-torus of $G$. Then the following three conditions are
equivalent:

\begin{enumerate}
\item[\rm (a)]
$k(G)/k(G)^G$ is pure {\rm(}respectively, stably pure{\rm)};

\item[\rm(b)]
$k(\g)/k(\g)^G$ is pure {\rm(}respectively, stably pure{\rm)};

 \item[\rm(c)] for every field extension $F/k$  and  every integral $(G, S)$-variety $X$
over $F$, the extension $F(X)/F(X)^G$ is pure {\rm(}respectively,
stably pure{\rm)}.

\end{enumerate}

The following two conditions are equivalent and are implied by the
previous conditions:

\begin{enumerate}
 \item[\rm (d)] for every field extension $F/k$ and every maximal $F$-torus $T$
of $\,G_F$, the $F$-va\-ri\-ety $G_F/T$ is rational
{\rm(}respectively, stably rational{\rm)} over $F$;

 \item[\rm (e)] the $K_{\rm gen}$-variety $G_{K_{\rm gen}}/T_{\rm gen}$ is
$K_{\rm gen}$-rational {\rm(}respectively, stably rational{\rm)}.
\end{enumerate}

If $\;G$ is quasisplit, then all five conditions are equivalent.
\end{thm}

\begin{proof} By Proposition \ref{prop.versal},
both $G$ and $\g$ are versal $(G, S)$-varieties, where $S$ is a
maximal $k$-torus of $G$. The equivalence of (a), (b) and (c) now
follows from Theorem~\ref{thm1'}(b).

The implication ${\rm(d)}\Rightarrow\!{\rm(e)}$ is clear. Let us
prove that ${\rm(e)}\Rightarrow{\rm(d)}$. By assumption, the generic
fibre of the quotient\footnote{This quotient exists as affine
$(G/N)$-scheme. This is a consequence of the fact that $H$ is a
(fibrewise connected) reductive $(G/N)$-group, the scheme $G\times_k
(G/N)$ is affine over $G/N$, and  $G/N$ is a scheme of
characteristic zero. For our purposes, it is enough to know this
over a dense open subset of $G/N$, and that is a consequence of the
existence of the quotient over the field $k(G/N)$.} of the
$(G/N)$-group scheme $G \times_k (G/N)$ by the maximal torus $H$
over $G/N$ is $k(G/N)$-rational.

By Lemma~\ref{lem.fibreing}(b) it follows that over a dense Zariski
open set of $G/N$, all fibres are rational. The $F$-torus $T$ may be
represented by an $F$-point of this open set. Hence the result.

The implication ${\rm(c)}\Rightarrow\!{\rm(d)}$ is immediate, since
$G_F/T$ is a $(G, S)$-variety over $F$.

 Let us show that
${\rm(d)}\Rightarrow\!{\rm(c)}$ if $G$ is quasisplit. Let $X$ be a
$(G, S)$-variety. We may assume that $X$ is a $(G,S)$-fibration.  By
Proposition~\ref{cor.steinberg1}(a) the quotient map $\pi \colon X
\to Y=X/G$ has a rational section $s \colon X/G \dasharrow X$. Let
$K=k(Y)$ be the function field of $Y$. Let $\eta$ denote the generic
point of $Y$ and let $X_{\eta}$ be the fibre of $\pi$ over $\eta$.
The $K$-variety $X_{\eta}$ is a $G_{K}$-homogeneous space with a
$K$-point, namely $s(\eta) \in X_{\eta}(K)$. Let $T \subset G_{K}$
be the stabilizer of  $s(\eta)$. This is a maximal $K$-torus in
$G_{K}$. The
 $G_{K}$-homogeneous
space $X_{\eta}$ is isomorphic to $G_{K}/T$.
 By our construction, $K(G_K/T) = K(X_{\eta}) = k(X)$ and $K = k(X)^G$.
By (d), $K(G_K/T)$ is rational (respectively, stably rational) over
$K$. Hence, $k(X)$ is rational (respectively, stably rational) over
$k(X)^G$.
\end{proof}

\begin{remark}\label{rem.isogeny} {}From the equivalence
between (a) and (b) we conclude that the answers to
Questions~\ref{q1}(a) and (b) stated in the Introduction depend
only on the isogeny class of the group~$G$.
\end{remark}

\begin{remark} \label{rem.bruhat}
If $G_{F}$ is split and the maximal $F$-torus $T \subset G_{F}$ is
also split then it easily follows from the Bruhat decomposition of
$G$ that the homogeneous space $G_F/T$  which appears in part (d) of
Theorem~\ref{thm.reduction} is $F$-rational; cf. the last paragraph
on p.~219 in~\cite{borel}. For general $T \subset G_{F}$, we shall
see in this paper that the quotient $G_{F}/T$ need not be
$F$-rational.
\end{remark}

\section{Reduction to the case where the group $G$ is simple
and simply connected}

In the previous section we have seen that the answers to
Questions~\ref{q1}(a) and (b) stated in the Introduction are the
same. Moreover, these answers remain unchanged if we replace $G$ by
an isogenous group.  In this section we  reduce these questions  for
a general connected split reductive group $G$ to the case where $G$
is split, simple and simply connected.

\begin{prop} \label{prop.reductive}
Let $G$ be a connected reductive algebraic group over $k$, $Z$ be a
central $k$-subgroup of $G$ and $\widetilde{G} = G/Z$. Denote the
Lie algebras of $\;G$ and $\widetilde{G}$ by $\g$ and
$\widetilde{\g}$ respectively. Then the following properties are
equivalent:
\begin{enumerate}
\item[\rm(i)] $k(\g)/k(\g)^G$ is pure
{\rm(}respectively, stably pure{\rm)};
\item[\rm(ii)] $k(\widetilde{\g})/k(\tilde{\g})^{\widetilde{G}}$ is
pure {\rm(}respectively, stably pure{\rm)}.
\end{enumerate}
\end{prop}

\begin{proof} We may assume without loss of generality that $Z$ is
the centre of $G$.

Since $Z$ acts trivially on $\g$, the adjoint action of $G$ on $\g$
descends to a $\widetilde{G}$-action, making $\g$ into a
$(\widetilde{G}, \widetilde{S})$-variety, where $\widetilde{S}$ is a
maximal torus in the semisimple group $\widetilde{G}$. Thus both
$\g$ and $\widetilde{\g}$ are  $(\widetilde{G},
\widetilde{S})$-varieties. The action of $\widetilde{G}$ on each of
$\g$ and $\widetilde{\g}$ is linear. By Corollary~\ref{cor.versal-c}
both are versal.  The desired conclusion now follows from
Theorem~\ref{thm1'}(b).
\end{proof}

Taking $Z$ to be the radical (connected centre) of $G$, we see that
Proposition~\ref{prop.reductive} reduces Questions~\ref{q1} to the
case where $G$ is semisimple. Proposition~\ref{prop.semisimple}
below further reduces it to the case where $G$ is  simple.

\begin{lem} \label{lem.semisimple}
Let $G = G_1 \times \dots \times G_n$, where each $G_i$ is a
connected reductive $k$-group.  Denote the Lie algebra of $G_i$ by
$\g_i$ and the Lie algebra of $\;G$ by $\g = \g_1 \times \dots
\times \g_n$.

\begin{enumerate}
\item[\rm(a)] If $k(\g_i)/k(\g_i)^{G_i}$ is   pure
{\rm(}respectively, stably pure{\rm)} for every $i = 1, \dots, n$,
then $k(\g)/k(\g)^G$ is  pure {\rm(}respectively, stably pure{\rm)}.

\item[\rm(b)] If  each $G_{i}$ is split and
$k(\g)/k(\g)^G$ is stably pure,  then $k(\g_i)/k(\g_i)^{G_i}$ is
stably pure for every $i = 1, \dots, n$.
\end{enumerate}
\end{lem}

\begin{proof} Denote the categorical
quotient map for the adjoint $G_i$-action by $\pi_i \colon \g_i \to
\g_i  \catq G_i$. Then the categorical quotient map for the adjoint
$G$-action is
\[ \pi = \pi_1 \times \dots \times \pi_n \colon  \g =  \g_1 \times \dots
\times  \g_n \rightarrow (\g_1  \catq  G_1) \times \dots \times
(\g_n  \catq G_n) = \g  \catq  G \, . \] Clearly if each $\g_i$ is
rational (respectively, stably rational) over $\g_i  \catq  G_i$
then $\g$ is rational (respectively, stably rational) over $\g \catq
G$. This proves part (a).

 Let us prove (b).
Suppose that $k(\g)/k(\g)^G$ is stably pure. By symmetry it suffices
to show that $k(\g_1)/k(\g_1)^{G_1}$ is stably pure.

Let $S_i$ be a split maximal torus in $G_i$. Then $S = S_1 \times
\dots \times S_n$ is a maximal torus in $G$.  Consider the $(G,
S)$-variety $X = \g_1 \times (G_2/S_2) \times \dots (G_n/S_n)$,
where the $G$-action on $X$ is the direct product of the adjoint
action of $G_1$ on $\g_1$ and the left translation action $G_i$ on
$G_i/S_i$. Clearly $X  \catq  G = \g_1 \catq  G_1$ and the quotient
map for $X$ is the composition of the projection $\pr_1 \colon X \to
\g_1$ to the first component and the quotient map $\pi_1 \colon \g_1
\to \g_1  \catq  G_1$:
\[
X = \g_1 \times (G_2/S_2) \times \dots \times (G_n/S_n)
\xrightarrow{{\rm \pr}_1}
  \g_1 \xrightarrow{\pi_1}
  \g_1 {\catq} G_1.
  \]
Because $G_{i}$ and $S_{i}$ are both split, each $G_i/S_i$ is
rational over $k$; see Remark~\ref{rem.bruhat}. Hence, $X$ is
rational over $\g_1$.

By Proposition~\ref{prop.versal}(c) $\g$ is a versal $(G,
S)$-variety. Since $k(\g)/k(\g)^G$ is stably pure,
Theorem~\ref{thm1'}(b) tells us that $X$ is stably rational over $X
\catq G \cong  \g_1  \catq G_1$. Consequently, $\g_1$ is stably
rational over $\g_1/G_1$, i.e., $k(\g_1)$ is stably pure over
$k(\g_1)^{G_1}$, as desired.
\end{proof}

\begin{prop} \label{prop.semisimple}
Let $G $ be a split semisimple group over $k$.  Let $G_1, \dots,
G_n$ denote  the simple components of the simply connected cover of
$G$. Denote the Lie algebras of $G, G_1, \dots, G_n$ by $\g, \g_1,
\dots, \g_n$ respectively.

\begin{enumerate}
\item[\rm(a)] The following properties are equivalent:
    \begin{enumerate}
    \item[\rm(i)] $k(\g)/k(\g)^G$ is stably pure;
\item[\rm(ii)]
$k(\g_i)/k(\g_i)^{G_i}$ is stably pure for every $i$.
\end{enumerate}
\item[\rm(b)] If $k(\g_i)/k(\g_i)^{G_i}$ is pure for
every $i$, then $k(\g)/k(\g)^G$ is pure.
\end{enumerate}
\end{prop}

\begin{proof}
The fields $k(\g)$, $k(\g)^G$, $k(\g_i)$ and $k(\g_i)^{G_i}$ remain
unchanged if we replace $G$ by its simply connected cover. Hence we
may assume
 $G = G_1 \times  \dots \times G_n$,
and the proposition follows from Lemma~\ref{lem.semisimple}.
\end{proof}

\section{Are homogeneous spaces of the form $G/T$ stably rational?}
\label{sectG/T}

The rest of this paper will be devoted to proving
Theorem~\ref{thm1}. Using Theorem~\ref{thm.reduction} and
Proposition~\ref{prop.reductive} we may assume without loss of
generality that $G$ is simply connected and restate
Theorem~\ref{thm1} in the following equivalent form.

\begin{thm} \label{thm1.modified}
Let $G$ be a split, simple, simply connected algebraic group
 over $k$
and let $T_{\rm gen}$ be the generic torus of $G$. {\rm(}Recall that
$T_{\rm gen}$ is defined over the field $K_{\rm gen} = k(G/N)$; see
{\rm\S\ref{thegenerictorus}}.{\rm)} Then the homogeneous space
$G_{K_{\rm gen}}/T_{\rm gen}$ is
\begin{enumerate}
 \item[\rm(a)] rational over $K_{\rm gen}$ if $\,G$ is  split of type
${\sf A}_n$ $(n \geqslant 1)$  or ${\sf C}_n$  $(n \geqslant 2)$;

\item[\rm(b)] not stably rational over $K_{\rm gen}$
if $\,G$  is not of type ${\sf A}_n$ $(n \geqslant 1)$  or ${\sf
C}_n$  $(n \geqslant 2)$ or
 ${\sf G}_{2}$.
 \end{enumerate}
\end{thm}

\begin{remark} \label{rem.G/T-vs-T}
For $k$ algebraically closed, the question  whether or not the
generic torus $T_{\rm gen}$ of $G$ is itself $K$-rational has been
studied in some detail. The (almost) simple groups whose generic
torus is (stably) rational are classified in~\cite{ll} (in type
${\sf A}$ only),~\cite{cor-k} (for simply connected or adjoint $G$
of all types) and~\cite{lpr} (for all $G$). A comparison of
Theorem~\ref{thm1.modified} and \cite[Theorem 0.1]{cor-k} shows that
there is no obvious relation between the (stable) rationality of the
homogeneous space $G_K/T_{\rm gen}$ and that of the generic torus
$T_{\rm gen}$. On the other hand, in view of the results of this
section one might wonder if the stable rationality of $G_K/T_{\rm
gen}$ is related to that of the dual torus $T_{\rm gen}^0$. We shall
return to this question in Appendix.
\end{remark}

Let $X$ be a $K$-variety. As usual, we let $\overline{K}$ denote the
algebraic closure of $K$ and set ${\overline X}:=X_{\overline{K}}$.
The ring $\overline{K}[X]$ and the abelian group $\Pic {\overline
X}$ then come equipped with the actions of the Galois group of
$\overline{K}$ over $K$.

\begin{thm}[\cite{col-k}]\label{prop4.4}
Let $K$ be a field of arbitrary characteristic, let $G$ be a
semisimple, simply connected linear algebraic $K$-group and let $T
\subset G$ be a $K$-torus. Denote the character lattice of $T$ by
$T^*$.
\begin{enumerate}

\item[\rm(a)]
If a $K$-variety $X_{\rm c}$ is a smooth compactification of
$X=G/T$, then there is an exact sequence of Galois lattices
$$0 \to P  \to \Pic{\overline X}_{\rm c} \to T^* \to 0$$
with $P $ a permutation lattice.

\item[\rm(b)] If $\;G/T$ is stably $K$-rational,
then there exists an exact sequence of Galois lattices
$$ 0 \to P_{2}  \to P_{1}  \to T^* \to 0$$
with $P_{2} $ and $P_{1} $ permutation lattices.

\item[\rm(c)] If $\;G/T$ is stably $K$-rational,
then $\Sha^{1}_{\omega}(K,T^*)=0.$
\end{enumerate}
\end{thm}

  \begin{proof}
For $K$ of characteristic zero,   this is an immediate consequence
of the more general result \cite[Theorem 5.1]{col-k}. This special
case is easier to prove.

We let here ${\overline K}$ denote a separable closure of $K$.
Associated to the $T$-torsor $G \to G/T=X$ there is a well known
exact sequence of Galois lattices (see
\cite[Prop.\;2.1.1]{descenteII}):
  $$ 0 \to \overline{K}[X]^{\times}/\overline{K}^{\times}  \to  \overline{K}[G]^{\times}/\overline{K}^{\times}  \to T^* \to \Pic {\overline X}
  \to \Pic {\overline G}.$$
  Since $G$ is  semisimple and simply connected, we have ${\overline K}^{\times}
=
   \overline{K}[G]^{\times}$ and $ \Pic {\overline G}=0$ (see \cite[Prop.\,1]{popov0}).
    We thus get ${\overline K}^{\times}
=  \overline{K}[X]^{\times}$
   and $T^*
   \xrightarrow{\cong}
    \Pic {\overline X}$
      (this is a special case of
     \cite[Theorem~4]{popov0} where $\Pic {\overline {G/H}}$
     for arbitrary subgroup $H$ is described).
     The open immersion of smooth $K$-varieties  $X \subset X_{\rm c}$
   gives rise to an exact sequence of Galois lattices
   $$0 \to  \overline{K}[X]^{\times}/{\overline K}^{\times}  \to \Div_{\infty}{\overline X}_{\rm c}
   \to \Pic {\overline X}_{\rm c} \to \Pic {\overline X} \to 0.$$
   Here $\Div_{\infty}{\overline X}_{\rm c}$ is the free abelian group on points of codimension 1
   of ${\overline X}_{\rm c}$ with support in the complement of ${\overline X} $.
This is a permutation lattice, call it $P$. All in all, we get an
exact sequence of torsion-free Galois lattices
   $$ 0 \to P  \to \Pic {\overline X}_{\rm c} \to T^* \to 0.$$
This proves (a).

If ${\rm char}(K)=0$, then $G/T$ admits a smooth
$K$-compactification. Statement (b) is then a consequence of (a) and
the well known fact that if the smooth, proper $K$-variety $X_{\rm
c}$ is stably $K$-rational, then the  Galois lattice
$\Pic({\overline X}_{\rm c})$ is a stably permutation lattice (see
\cite[Prop. 2.A.1]{descenteII}). Statement (c) then  follows from
(b) (see  \S \ref{sect2}).

There is however no need to use a smooth compactification to prove
(b) and (c). If  $X=G/T$ is stably $K$-rational, there exist natural
integers $r$ and $s$ and dense open sets $U \subset Y=X \times_{K}
\bbA^r_{K}$ and $V \subset \bbA^s_{K}$ together with a
$K$-isomorphism $U \oii V$. The natural maps ${\overline K}^{\times}
\to {\overline K}[\bbA^s]^{\times}$ and $  {\overline K}[X]^{\times}
\to {\overline K}[X \times_{K} \bbA^r]^{\times}$
 are isomorphisms.  By what we have seen above,
$ {\overline K}^{\times} \to {\overline K}[X]^{\times}$ is an
isomorphism. Thus $ {\overline K}^{\times} \oii {\overline
K}[Y]^{\times}$. We have $\Pic {\overline \bbA}^s=0$, hence
$\Pic{\overline V}=0$, hence $\Pic{\overline U}=0$. The pull-back
map associated to the projection $Y=X \times_{K} \bbA^r \to X$
induces an isomorphism of Galois modules $\Pic{\overline X}
 \oii  \Pic{\overline Y}$.
 From what we have proved above this induces an isomorphism
$T^* \simeq \Pic{\overline Y}$.

The open immersion $V\subset \bbA^s_{K}$ induces an isomorphism of
Galois modules between the permutation module on irreducible
divisors of $\bbA^s_{\overline K}$ with support in the complement of
${\overline V}$ and the Galois module ${\overline
K}[V]^{\times}/{\overline K}^{\times}$, hence the Galois module
${\overline K}[U]^{\times}/{\overline K}^{\times}$ is a permutation
module. The open immersion $U \subset Y$ induces an exact sequence
$$0 \to {\overline K}[U]^{\times}/{\overline K}^{\times} \to \Delta \to \Pic{\overline Y} \to 0,$$
where $\Delta$ is the permutation module on irreducible divisors of
${\overline Y}$ with support in the complement of ${\overline V}$.
This completes the proof of (b), hence also of (c).
 \end{proof}

 \begin{remark}\label{Brauergp} In certain circles, the (unramified) Brauer
group is a well known $K$-birational invariant of smooth,
projective, geometrically integral $K$-varieties. For $X \subset
X_{\rm c}$ as above, this is the group $\Br X_{\rm c}$.
 The connection with
the above proposition is given by an isomorphism $ \Br X_{\rm c}/\Br
K \oii     H^1(K, \Pic  {\overline X}_{\rm c}) $, which one combines
with an isomorphism $ H^1(K, \Pic  {\overline X}_{\rm c}  ) \oii
\Sha^{1}_{\omega}(K,T^*)$ deduced from Statement (a) to produce an
isomorphism
 $$\Br X_{\rm c}/\Br K  \oii \Sha^{1}_{\omega}(K,T^*).$$
 The interested reader is referred to
\cite{col-k} for details.
 \end{remark}

 Under a strong assumption on the group $G$, we shall now establish a converse
 to statement (b) in Theorem~\ref{prop4.4}.
 We first prove a lemma.

\begin{lem} \label{lem3.1} Let $K$ be a field of
characteristic $0$ and let $H$ be a special linear algebraic
$K$-group. If $X$ is a geometrically integral $K$-variety with a
generically free action of $H$, then $X$ admits a dense $H$-stable
open set $U$ which is $H$-isomorphic to   $H \times_{K} Y$ for a
$K$-variety $Y$ whose function field $K(Y)$ is $K$-isomorphic to
$K(X)^H$ and for the $H$-action induced by left translation of $H$.
\end{lem}

\begin{proof}  After replacing $X$ by a friendly open set, we may
assume that $X$ is the total space of an $H$-torsor $X \to Y=X/H$.
Since $H$ is special, this torsor splits over a dense open set of
$Y$. If we replace $Y$ by this open set, we have $X=H \times_K Y$ as
varieties with an $H$-action.
\end{proof}

\begin{prop}\label{prop4.4bis}
Let $K$ be a field of characteristic zero, and let $G $ be a
special, $K$-rational $K$-group. Let $T \subset G$ be a  $K$-torus.
If there exists an exact sequence of Galois modules
$$ 0 \to P_{2}  \to P_{1}  \to T^* \to 0$$
with $P_{2} $ and $P_{1} $ permutation modules,
 then $G/T$ is  stably $K$-rational.
\end{prop}

\begin{proof}
By assumption we  have an exact sequence of $K$-tori
\[ 1 \to T \to T_{1} \to T_{2} \to 1 \]
with $T_{1}$ and $T_2$ quasitrivial. We now identify $T$ with the
diagonal subgroup $T$ of  $ G \times T_{1}$, and $G$ with the
subgroup $G \times \{ 1 \}$ of $G \times T_1$. The first projection
$G \times T_1 \to G$ induces a map $(G \times T_1)/T \to G/T$ which
makes $(G \times T_1)/T$ into a right $T_{1}$-torsor over $G/T$. The
second projection $G \times T_1 \to T_{1}$ induces a map $(G \times
T_1)/T \to T_{1}/T$ which makes $(G \times T_1)/T$ into a left
$G$-torsor over $T_{1}/T = T_{2}$. In summary, we have the following
diagram:
\[ \xymatrix@C=3mm@R=6mm{  &  (G \times T_1)/T
\ar@{->}[dr]^{\hskip 4mm\text{\footnotesize left $G$-torsor}}
\ar@{->}[dl]_{\text{\footnotesize right $T_1$-torsor}\hskip 2mm} &
\cr G/T & & T_1/T = T_2 \cr } \] Since $T_{1}$ is a quasitrivial
torus hence a special $K$-group, $(G \times T_1)/T$ is
$K$-birationally isomorphic  to $T_{1} \times (G/T)$; see
Lemma~\ref{lem3.1}. Similarly, since $G$ is special, $(G \times
T_1)/T$ is $K$-birationally isomorphic to $G \times T_{2}$.  Thus
$T_{1} \times (G/T)$ is $K$-birationally isomorphic to $G \times
T_{2}$. Since $T_{1}, T_{2}$ and $G$ are $K$-rational varieties, we
conclude that $G/T$ is stably $K$-rational.
\end{proof}

We now specialize Theorem~\ref{prop4.4} and
Proposition~\ref{prop4.4bis} to the setting of
Theorem~\ref{thm1.modified}.

\medskip

  The weight lattice $P({\sf R})$ of a root system of type ${\sf R}$
  is equipped with a natural action
of the Weyl group $W$.

\begin{cor} \label{cor.generic}
Let $G$ be a split, simple, simply connected algebraic group of type
${\sf R}$, defined over $k$ and $T_{\rm gen}$ be the generic torus
of $\;G$. Recall that $T_{\rm gen}$ is defined over the field
$K_{\rm gen} = k(G/N)$; cf. {\rm\S\ref{thegenerictorus}}.

\begin{enumerate}
 \item[\rm(a)] If the homogeneous space $G_{K_{\rm gen}}/T_{\rm gen}$
is stably rational over $K_{\rm gen}$, then there exists an exact
sequence
 \begin{equation}\label{PPP}
 0 \to P_2 \to P_1 \to P({\sf R}) \to 0
\end{equation}
of $\;W$-lattices, where $P_1$ and $P_2$ are permutation.

\item[\rm (b)] Suppose $G$ is special. Then the
converse to part {\rm(a)} holds. That is, if there exists an exact
sequence \eqref{PPP} of $\;W$-lattices with $P_1$ and $P_2$
permutation lattices, then $G_{K_{\rm gen}}/T_{\rm gen}$ is stably
rational over $K_{\rm gen}$.
\end{enumerate}
\end{cor}

\begin{proof} As recalled in \S \ref{thegenerictorus}, the $K_{\rm gen}$-torus
$T_{\rm gen}$ splits over a  Galois extension of $K_{\rm gen}$ with
Galois group the Weyl group $W$, and the character lattice of
$T_{\rm gen}$ is isomorphic to the weight lattice $P({\rm R})$ with
its natural $W$-action. Note also that since $G$ is split, $G$ is
rational over $k$ and hence $G_{K_{\rm gen}}$ is rational over
$K_{\rm gen}$. Applying Theorem~\ref{prop4.4} (b) and
Proposition~\ref{prop4.4bis}   to $G_{K_{\rm gen}}$ and $T_{\rm
gen}$ and using the correspondence between tori and lattices
(\S\ref{sect2}) we get (a) and (b).
\end{proof}

\section{Nonrationality}
\label{sect.non-rationality}

In this section we shall prove Theorem~\ref{thm1}(b) or
equivalently, Theorem~\ref{thm1.modified}(b). To prove the latter,
one may assume that $k$ is algebraically closed. In view of
Corollary~\ref{cor.generic}(a), it suffices to establish the
following proposition.

\begin{prop} \label{prop5.3}
Let ${R}$ be a reduced, irreducible root system in a real vector
space $V$ that it spans.
 Let $P({R})$ be the weight lattice equipped with the action of the
Weyl group $W=W({R})$. If ${R}$ is not of type ${\sf A}_{n}$, ${\sf
C}_{n}$ or ${\sf G}_{2}$, then there exists a subgroup $H \cong
(\bbZ/2\bbZ)^2 $ in $W$ such that $\Sha^1_{\omega}(H,P({R})) \neq
0.$ In particular, there does not exist an exact sequence of
$W$-lattices
$$0 \to P_{2} \to P_{1} \to P({R}) \to 0 $$
with $P_{1}$ and $P_{2}$ permutation lattices.
\end{prop}

\begin{proof}
Let $B$ be a basis of ${R}$. Let $W=W({R})$ be the Weyl group. The
abelian group $Q({R}) \subset V$ spanned by ${R}$ is the root
lattice, its rank is $l:= \dim(V)$. There is an inclusion $Q({R})
\subset P({R}) \subset V$, where $P({R})$ is the weight lattice. See
\cite[VI.\,1.9]{Bou}. Both $Q({R})$ and $P({R})$ are $W$-lattices.

Let $B' \subset B$ be a subset of $B$ of cardinality $l'$, let $V'
\subset V$ be the vector space spanned by $B'$ and let ${R}'={R}
\cap V'$. This $R'$ is a root system in $V'$, $B'$ is a basis of
${R}'$, and $Q({R}')=Q({R}) \cap V'$. See \cite[VI.1.7, Cor.\,4,
p.\,162]{Bou}.

This implies that $Q ({R}')$ is a direct factor of $Q ({R})$ (as an
abelian group). Moreover, since $W ({R}')$ is generated by the
reflections in the hyperplanes orthogonal to the roots $\alpha \in
{R}'$, $W ({R}')$ can naturally be viewed as a subgroup of $W
({R})$. {}From the formula $s_{\alpha}(\beta)=\beta -
n_{\beta,\alpha} \alpha$ we see that $W({R}')$ acts trivially on
$V/V'$ and hence on $Q({R})/Q({R}')$. We write $Q ({R})/Q({R}')=\bbZ
^{l-l'}$, the trivial $W ({R}')$-lattice. In other words, there is a
short exact sequence of $W({R}')$-lattices
\begin{equation}\label{Qsequence}
  0 \to Q({R}') \to Q({R}) \to {\bbZ}^{l-l'} \to 0.
  \end{equation}
Our proof of Proposition~\ref{prop5.3} will rely on the following
claim.

\begin{claim} \label{claim1}
For a root system ${R}$ dual  {\rm(}inverse in the terminology of
{\rm\cite[VI\,1.1]{Bou}}{\rm)} to one of the root systems occurring
in the statement of the proposition, there exist  a subsystem ${R}'
\subset {R}$ {\rm(}as above{\rm)}, a subgroup $H \cong
(\bbZ/2\bbZ)^2 $ in $W({R}')$ and a direct factor $J_H$ of the
$H$-lattice $Q({R}')$, where $J_H$ is the cokernel of the map $\bbZ
\to \bbZ[H]$ given by the norm.
\end{claim}

Indeed, assume Claim~\ref{claim1} is established. Consider the exact
sequences
\begin{gather*}
0 \to \bbZ \to \bbZ[H] \to J_H \to 0,\\
0 \to I_H \to \bbZ[H] \to \bbZ \to 0,
\end{gather*}
where the map $\bbZ[H] \to \bbZ$ is augmentation. The latter
sequence yields $ \Sha^1_{\omega}(H,I_H) \cong\bbZ/2\bbZ$.

The Weyl groups of a root system ${R}$ and of its dual ${R}^{\vee}$
are identical. Exact sequence (\ref{Qsequence})
 induces an exact
sequence of $W(R')$-lattices (the last two are weight lattices)
\begin{equation}\label{Psequence}
0 \to {\bbZ}^{l-l'} \to P({R}^{\vee}) \to P({R}'^{\vee}) \to 0,
 \end{equation}
which we view as an exact sequence of $H$-lattices. Here
${R}^{\vee}$ is a root system as occurring in the proposition to be
proved, i.e., a root system not of type ${\sf A}_{n}$, ${\sf C}_{n}$
or ${\sf G}_{2}$. (Recall that by the assumption $G$ is simply
connected, hence the $W$-lattice given by the character group of a
maximal torus   is the weight lattice.)

Using Claim~\ref{claim1} we conclude that the $H$-lattice
$P({R}'^{\vee})$ (dual to $Q({R}')$) contains the $H$-lattice $I_H$
(dual to $J_H$) as a direct factor, hence
$\Sha^1_{\omega}(H,P({R}'^{\vee})) \neq 0$. {}From the exact
sequence (\ref{Psequence})  we get, by a standard computation,
$$\Sha^1_{\omega}(H,P({R}^{\vee})) \cong \Sha^1_{\omega}(H,P({R}'^{\vee}))$$
hence
$$ \Sha^1_{\omega}(H,P({R}^{\vee}))\neq 0.$$

To complete the proof of Proposition~\ref{prop5.3} it remains to
establish Claim~\ref{claim1}.

\smallskip
{\em Proof of Claim~{\rm\ref{claim1}}}: The root systems ${R}$ dual
to those considered in the proposition are those of types $ {\sf
C}_n \, (n \geqslant 3)$, ${\sf D}_n \, (n \geqslant 4)$, ${\sf E}_r
\, (r = 6,7,8)$ and ${\sf F}_4$.

Any Dynkin diagram of type ${\sf C}_n$\,$(n \geqslant 3)$ contains a
subdiagram of type ${\sf C}_3$. All the other ones in the list
above, except ${\sf F}_4$, contain a subdiagram of type ${\sf D}_4$.
The case of ${\sf F}_4$ can be reduced to ${\sf D}_4$ because the
weight lattices $P({\sf F}_4)$ and  $P({\sf D}_4)$ coincide (as
abelian groups in $\mathbb{R}^4$) and $W({\sf D}_4)\subset W({\sf
F}_4) \subset {\rm GL}_4(\mathbb{R})$ (compare Planche IV and
Planche VIII in \cite{Bou}).

For the reader's convenience, we reproduce some of the calculations
from~\cite{cor-k}.

Recall that $W({\sf C}_n)$ is a semidirect product $(\bbZ /2\bbZ
)^n\rtimes S_n$. We denote by $c_1,\dots ,c_n$ the natural
generators of $(\bbZ /2\bbZ)^n$.

Let us first discuss the case where $R$ is of type ${\sf C}_{3}$. We
choose $H=\left< c_1c_3,c_2(13)\right> =\left< a,b\right> \subset
W({\sf C}_3)$. In the basis $\alp _1=\eps _1-\eps _2$, $\alp _2=\eps
_2-\eps _3$, $\alp _3=\eps _2+\eps _3$, the group $H$ acts on
$M=Q({\sf C}_3)$ as follows:
$$
a\colon \left\{\hskip -1mm \begin{array}{ccl}
\alp _1  & \hskip -2mm \mapsto  & \hskip -2mm -\alp _1-\alp _2-\alp _3,  \\
\alp _2  & \hskip -2mm \mapsto  & \hskip -2mm  \alp _3,  \\
\alp _3  & \hskip -2mm \mapsto  & \hskip -2mm  \alp _2,
\end{array} \right.
\qquad b\colon \left\{\hskip -1mm \begin{array}{ccl}
\alp _1  & \hskip -2mm \mapsto  & \hskip -2mm  \alp _3,  \\
\alp _2  & \hskip -2mm \mapsto  & \hskip -2mm -\alp _1-\alp _2-\alp _3, \\
\alp _3  & \hskip -2mm \mapsto  &  \hskip -2mm \alp _1.
\end{array} \right.
$$
This coincides with the standard formulas for $J_H$.

Let us now discuss the case where $R$ is of type ${\sf D}_{4}$. In
$\mathbb R ^4$ equipped with the standard basis $\eps _1,\dots ,\eps
_4$, we consider $M=Q({\sf D}_4)$ with $\bbZ$-basis $\alp _1=\eps
_1-\eps _2$, $\alp _2=\eps _2-\eps _3$, $\alp _3=\eps _3-\eps _4$,
$\alp _4=\eps _3+\eps _4$. The Weyl group $W({\sf D}_4)$ can be
identified with the subgroup of $W({\sf C}_4)$ consisting of the
elements with even numbers of $c_i$'s. We choose $H=\left< c_3c_4,
c_1c_2(34)\right>$. The group $H$ acting on $M$ respects
$V'=\left<\eps _2,\eps _3, \eps _4\right> = \left<\alp _2, \alp _3,
\alp _4\right>$, and $R'= R\cap V'$ is of type ${\sf D}_3$.
Moreover, $H$ respects the one-dimensional $\bbZ$-module generated
by $\alp _1$: indeed, $c_3c_4$ fixes $\alp _1$ and $c_1c_2(34)$
sends $\alp _1$ to $-\alp _1$. Therefore the $H$-lattice $M$
decomposes into a direct sum of a one-dimensional lattice and a
three-dimensional lattice. It remains to note that the latter
three-dimensional lattice $J$ is isomorphic to $J_H$. To see that,
we observe that the action of $c_1c_2(34)$ on $J$ coincides with the
action of $c_2(34)\in W({\sf C}_3)$ on $Q({\sf C}_3)=Q({\sf D}_3)$,
and we are led (up to permutation of indices) to the former case.

This completes the proof of Claim~\ref{claim1}, hence of
Proposition~\ref{prop5.3}, hence of Theorems~\ref{thm1.modified}(b)
and~\ref{thm1}(b).
\end{proof}

\begin{remark} Our proof of Theorem~\ref{thm1.modified}(b)
actually establishes the following stronger assertion.

\begin{prop} Let $G$ be a simple simply connected linear
algebraic group over $k$ which is not of type ${\sf A}_{n}, {\sf
C}_{n}$ or ${\sf G}_{2}$. Let $T_{\rm gen}$ be the generic torus of
$\,G$. Recall that $T_{\rm gen}$ is defined over the field $K_{\rm
gen} = k(G/N)$; cf. {\rm\S\ref{thegenerictorus}}. Then $(G_{K_{\rm
gen}}/T_{\rm gen}) \times_{K_{\rm gen}} Y$ is not rational over
$K_{\rm gen} $ for any $K_{\rm gen} $-variety $Y$.
\end{prop}

Indeed, let $X_{\rm c}$ be a smooth $K_{\rm gen} $-compactification
of $X = G_{K_{\rm gen} }/T_{\rm gen}$. Combining
Theorem~\ref{prop4.4}, Remark \ref{Brauergp}  and
Proposition~\ref{prop5.3} we find that there exists a finite field
extension $M/K_{\rm gen} $ such that $\Br (X_{{\rm c}})_ M/\Br M
\neq 0$. On the other hand, if there exists a $K_{\rm gen} $-variety
$Y$ such that $X \times_{K_{\rm gen} }Y$ is $K_{\rm gen}
$-birationally isomorphic  to projective space, then $\Br (X_{{\rm
c}})_ M/\Br M =0$ for any field extension $M/K_{\rm gen} $. As a
matter of fact, the nonvanishing of $\Br (X_{{\rm c}})_ M/\Br M$
implies that the $K_{\rm gen} $-variety $X=G_{K_{\rm gen} }/T_{\rm
gen}$ is not even retract rational (a concept due to D. Saltman);
see \cite[\S 1, Prop. 5.7 and Remark 5.8]{chili}. \qed
\end{remark}

\section{Weight lattices for root systems of types ${\sf A}_n$,
${\sf C}_n$, and ${\sf G}_2$} \label{sect5}

In this section we shall prove the following converse to
Proposition~\ref{prop5.3}.

\begin{prop} \label{prop.weight-lattice}
Let ${R}$ be a   reduced, irreducible root system and $P({R})$ be
the weight lattice equipped with the action of the Weyl group
$W=W({R})$. If  $R$ is of type ${\sf A}_n $, ${\sf C}_n$ or ${\sf
G}_2$, then there exists an exact sequence of $\,W$-lattices
$$0 \to P_{2} \to P_{1} \to P({R}) \to 0 $$
with $P_{1}$ and $P_{2}$ permutation lattices.
\end{prop}

\begin{proof} Suppose ${R}$ is of type  ${\sf G}_2$.  Here
 the $W$-lattice $M = P({\sf G}_2)$ is of rank $2$.
Thus, as we pointed out at the end of Section~\ref{sect2}, $M$ fits
into an exact sequence
\[ 0 \to P_{2} \to P_{1} \to M \to 0 \, , \]
with $P_1$ and $P_2$ permutation.

Now suppose ${R}$ is of type ${\sf A}_{n}$. Then $W=S_{n+1}$. {}From
the Bourbaki tables~\cite{Bou} we get the exact sequence of
$S_{n+1}$-lattices
$$\textstyle 0 \to Q({\sf A}_{n}) \to  \bigoplus_{{i=1}}^{n+1} \bbZ
\varepsilon_{i} \to \bbZ \to 0, $$ where the action of $S_{n+1}$ on
the middle term is by permutation, the action on the right hand side
$\bbZ$ is trivial and the right hand side map to $\bbZ$ is
augmentation, i.e., summation of coefficients.

If one dualizes this sequence one gets an exact sequence of
$S_{n+1}$-lattices
\begin{equation}\label{A_{n}sequence}
\textstyle 0 \to  \bbZ  \to \bigoplus_{{i=1}}^{n+1} \bbZ
\varepsilon_{i} \to P({\sf A}_{n}) \to 0,
\end{equation}
where the action of $S_{n+1}$ on the middle term is by permutation,
the action on the left hand side $\bbZ$ is trivial and the map with
source $\bbZ$
 sends $1$
to the sum of the $\varepsilon_{i}$.

Finally suppose ${R}$ is of type ${\sf C}_{n} \, (n \geqslant 2)$.
Then $W$ is the semidirect product of $S_{n}$ by $ (\bbZ/2\bbZ)^n$.
Let us first look at the ${\sf B}_{n}$-table in~\cite{Bou}. There is
an exact sequence of $W({\sf B}_{n})$-lattices
$$\textstyle 0 \to Q({\sf B}_{n}) \to
\bigoplus_{i=1}^n (\bbZ a_{i } \oplus \bbZ b_{i}) \to
\bigoplus_{i=1}^n \bbZ c_{i} \to 0,$$ where $a_{i}$  and $b_{i}$ are
mapped to $c_{i}$, and the action of $W$ is as follows.  On the
right hand lattice $\bigoplus_{i=1}^n \bbZ c_{i}$, the action is the
permutation action of the quotient $S_{n}$. On the middle lattice,
$S_{n}$ acts by naturally permuting the $a_{i}$ and the $b_{i}$.  An
element $(\alpha_{1}, \ldots, \alpha_{n}) \in (\bbZ/2\bbZ)^n$ fixes
$a_{i}$ and $b_{i}$ if $\alpha_{i}=0$ and it permutes them if
$\alpha_{i}=1$.
   In Bourbaki's notation for
${\sf B}_{n}$, we have $\varepsilon_{i}=a_{i}-b_{i}$. If one
dualizes the above sequence, one gets the exact sequence of $W({\sf
C}_{n})$-lattices
\begin{equation}\label{C_{n}sequence}
\textstyle 0 \to \bigoplus_{i=1}^n \bbZ \gamma_{i} \to
\bigoplus_{i=1}^n (\bbZ \alpha_{i } \oplus \bbZ \beta_{i}) \to
P({\sf C}_{n}) \to 0,
\end{equation}
where the two left lattices are permutation lattices.
 \end{proof}

\begin{prop}\label{stablerationalityAC}
Let $G$ be a split, simply connected  simple group of type ${\sf
A}_n$ or ${\sf C}_n$ {\rm(}i.e., $G = \SL_n$ or $\Sympl_{2n})$. Then
the field extensions $k(G)/k(G)^G$ and $k(\g)/k(\g)^G$ are stably
pure.
\end{prop}
\begin{proof}
Since these groups are special, Propositions ~\ref{prop4.4bis}
and~\ref{prop.weight-lattice} imply that $G_K/T_{\rm gen}$ is stably
rational over $K_{\rm gen}$ (or equivalently,   $k(\g)$ is stably
rational over $k(\g)^{G}$). The statement then follows from Theorem
\ref{thm.reduction}.
\end{proof}

\begin{remark}
This is weaker than the rationality assertion of
Theorem~\ref{thm1.modified}(a) (or equivalently, of
Theorem~\ref{thm1}(a)), which will be proved in the next section. In
the meantime, we remark that the same argument cannot be used to
show that $G_K/T_{\rm gen}$ is stably rational (or equivalently,
$k(\g)$ is stably rational over $k(\g)^{G}$) for the split ${\rm
G}_2$, because this group is not special and
Proposition~\ref{prop4.4bis} does not apply to it. In fact, in this
case we do not know whether or not $G_K/T_{\rm gen}$ is stably
rational over $K_{\rm gen}$.
\end{remark}

\section{Rationality}
\label{sect.rat}

We now turn to the proof of Theorem~\ref{thm1.modified}(a) (or
equivalently, of Theorem~\ref{thm1}(a)). Our argument will be based
on the rationality criterion of Lemma~\ref{lem3.15}(c) below. For
groups of type ${\sf A}_n$, Theorem~\ref{thm1.modified}(a) will be
an easy consequence of this criterion. For groups of type ${\sf
C}_n$ the proof proceeds along the same lines but requires a more
elaborate argument.

Let $K$ be a field of characteristic zero,  let $G$ be a linear
algebraic group over $K$, and let $H_1, H_2$ be closed
$K$-subgroups. The actions of $H_1$ and $H_2$ on $G$ by respectively
left and right translation commute, thus giving rise to an $(H_1
\times H_2)$-action on $G$. The action of $H_{1}$ on $G$ defines an
$H_{1}$-torsor $\pi_{1} \colon G \to  H_{1} \ssetminus G$, and the
action of $H_{2}$ defines an $H_{2}$-torsor $\pi_{2} \colon G \to
G/H_{2}$.

Using Rosenlicht's theorem (see \S \ref{sect.quotients}), one may
find a $(H_1 \times H_2)$-stable dense open subset $U$ of $G$ such
that the action of $H_1 \times H_2$ on $U$ mods out to a {\it
geometric quotient} $U \to V$ which factorizes through $U \to U_{1}
\to V$ and $U \to U_{2} \to V$, where $U_{1 } \subset H_{1}
\ssetminus G$ and $U_{2} \subset G/H_{2}$ are open sets and $U \to
U_{1}$, respectively,\;$U \to U_{2}$ is an $H_{1}$-torsor,
respectively,\;an $H_{2}$-torsor.

In this section we shall indulge in  the following notation. We
shall adopt the double coset notation $H_1 \ssetminus G/ H_2$ for
{\it some} $V$ as above. In  particular, we have
\[ K(H_1 \ssetminus G/H_2) \cong  K(G)^{H_1 \times H_2} \,. \]
We have a commutative diagram of rational maps, where, as usual,
solid arrows denote regular maps and dotted arrows denote rational
maps:
\begin{equation*}
\begin{matrix}
\xymatrix@R=3.4mm{  &  G \ar@{->}[dl]_{\pi_1} \ar@{-->}[dd]^{\pi}
\ar@{->}[dr]^{\pi_2} & \cr H_1\ssetminus G \ar@{-->}[dr] & &  G/H_{2
} \ar@{-->}[dl]     \cr
 & H_1 \ssetminus G/H_2
 & \cr }
 \end{matrix}\quad .
 \end{equation*}

Let $Z(G)$ denote the centre of $G$.

\begin{lem} \label{lem3.15}
Let $K$ be a field of characteristic zero and let $G$ be a simple
$K$-group. Suppose $G$ has subgroups $H$ and $T$, where $T$ is a
torus and $H \cap Z(G) = \{ 1\}$. Then

\begin{enumerate}
 \item[\rm(a)] the
 $H$-action on $G/T$ is generically free;

\item[\rm(b)] $\dim(H \ssetminus G /T) = \dim(G) -
\dim(H) - \dim(T)$;

\item[\rm(c)] if $H$ is special  and both $H$, $H
\ssetminus G/T$ are $K$-rational, then $G/T$ is also $K$-rational.
\end{enumerate}
\end{lem}

\begin{proof}
To prove (a), we may pass to an algebraic closure $\overline{K}$ of
$K$ and thus assume, without loss of generality, that $K$ is
algebraically closed.

Note that conditions

\begin{enumerate}
\item[(a)] the left $H$-action on $G/T$ is generically
free;

\item[(${\rm a}^\prime$)] the right $T$-action on $H \!
\setminus \! G$ is generically free

\end{enumerate}
\noindent are equivalent. Indeed, (a) says that $H \cap g T g^{-1} =
\{ 1 \}$ for $g \in G$ in general position and (${\rm a}^\prime$)
says that $T \cap g H g^{-1} = \{ 1 \}$ for $g \in G$ in general
position. Thus (a) and (${\rm a}^\prime$) are equivalent, and it
suffices to prove (${\rm a}^\prime$).

 Assume to the contrary that the $T$-action on $H \!
 \setminus \! G$ is not generically free.
 By a result of Sumihiro \cite{sumihiro}
 there is an affine $T$-stable dense open set $U$ of $H \!
 \setminus \! G$. By the Embedding Theorem \cite[Theorem\,1.5]{popov-vinberg},
  $U$ is a $G$-stable closed irreducible subvariety
 of a finite-dimensional $T$-module $V$ not contained in
a proper $T$-submodule of $V$. Hence $U$  intersects the complement
to the union of weight spaces of $T$. But the stabilizer of a
point in this complement coincides with the kernel of the action of
$T$ on $V$. Thus the action of $T$
 on $U$, hence on $H \!
 \setminus \! G$
 has a nontrivial kernel $\Gamma \subset T$, cf.\;\cite[Sect.\,7.2, Prop.]{popov-vinberg}
 Then $\Gamma$ is contained in $N = \bigcap_{g
 \in G} gHg^{-1}$, which is a normal subgroup of $G$.
 Since  $N \subsetneq G$ and we are assuming that $G$ is simple,
 we conclude that $N \subset Z(G)$.
 Thus $\{ 1 \} \subsetneq  \Gamma \subset N = N \cap Z(G)
 \subset H \cap Z(G)$, contradicting our assumption
 that $H \cap Z(G) = \{ 1 \}$. This contradiction
 proves part (a).

\smallskip
By Theorem \ref{generalGS} (case $S=\{1\}$) and the standard formula
for the dimension of a variety fibred over another variety, (b)
follows from (a).

\smallskip
Let us now prove (c).  Part (a) allows us to apply
Lemma~\ref{lem3.1}
 to the left $H$-action on $G/T$. By Lemma~\ref{lem3.1}, $G/T$ is
$K$-birationally isomorphic to $H \times (H \ssetminus G/T)$. Since
we are assuming that both $H$ and $H \ssetminus G/T$ are
$K$-rational, so is $G/T$.
\end{proof}

We now proceed with the proof of Theorem~\ref{thm1.modified}(a). The
group $G$ is a split simply connected simple group of type ${\sf
A}_n$ or ${\sf C}_n$, over the field $k$. In the sequel we will set
$K = K_{\rm gen}$ and will work with the generic $K$-torus $T=T_{\rm
gen} \subset G_{K}$, as in~\S\ref{thegenerictorus}.

 \medskip
{\bf Type ${\sf A}_n$.} Let $H$ denote  the stabilizer of a non-zero
element for the natural action of $G= \SL_{n+1}$ on ${\mathbb
A}_{k}^{n+1}$. It is easy to see that $H$ is isomorphic to a
semidirect product $U \rtimes \SL_n$, for some unipotent group $U$
and that $H \cap Z(\SL_{n+1})=1$. By Lemma~\ref{lem3.15}(c), it
suffices to show that (i) $H$ is special, (ii) $H$ is $k$-rational,
and (iii) the ``double coset space" $H_K \ssetminus G_K /T_{\rm
gen}$ (whose definition is explained above) is $K$-rational. We now
proceed to prove (i), (ii) and (iii).

(i) For any field extension $F/k$, the natural map
\[ H^1(F, H) \to H^1(F, \SL_{n}) \]
is an isomorphism; see, e.g., \cite[Lemme 1.13]{sansuc}. Since
$\SL_{n}$ is a special group, we conclude that $H^1(F, H) \cong
H^1(F, \SL_n) = \{ 1 \}$, i.e., $H$ is special.

(ii) In characteristic zero, any unipotent group is special
(see~\cite[Prop.\;III.2.1.6]{serre-gc}) and rational (see,
e.g.,~\cite[Example 1.21]{lpr}). Viewing the natural projection $H
\to \SL_n$ as a $U$-torsor over $\SL_n$, we see that $H$ is
$k$-birationally isomorphic to $U \times \SL_n$. This shows that $H$
is $k$-rational.

(iii) $H_K \ssetminus G_K /T_{\rm gen}$ is a 1-dimensional
$K$-variety; see Lemma~\ref{lem3.15}(b). It is clearly unirational
over $K$ (it is covered by $G_K$). By L\"uroth's theorem  it is thus
$K$-rational.

\smallskip
This completes the proof of Theorem~\ref{thm1.modified} for groups
of type ${\sf A}_n$.

\medskip
{\bf Type ${\sf C}_n$}. Let  $G :=\Sympl_{2n}$. Once again, we let
$H$ be the $\Sympl_{2n}$-stabilizer of a non-zero vector $v \in
k^{2n}$ for the natural action of $G$ on ${\mathbb A}^{2n}$. It is
well known that $H$ is $k$-isomorphic to a semidirect product of $U
\rtimes \Sympl_{2n-2}$, where $U$ is a unipotent group defined over
$k$; see, e.g.,~\cite[pp.\,35--36]{weil} or~\cite[p.\,384]{igusa}.
Once again, by Lemma~\ref{lem3.15}(c), it suffices to show that
\begin{enumerate}
\item[(i)] $H$ is special; \item[(ii)] $H$ is
$k$-rational;
\item[(iii)] $H_K \ssetminus G_K
/T_{\rm gen}$ is $K$-rational.
    \end{enumerate}
\noindent The proofs of (i) and (ii) are exactly the same as for
type ${\sf A}$. In order to complete the proof of
Theorem~\ref{thm1.modified}(a) (or equivalently, of
Theorem~\ref{thm1}(a)), it thus remains to establish (iii), which we
now restate as a proposition.

\begin{prop} \label{prop2.type-c}
$H_K \ssetminus G_K /T_{\rm gen}$ is $K$-rational.
\end{prop}

To prove the proposition, we first note that $H \ssetminus G$ is, by
definition, the $\Sympl_{2n}$-orbit of a non-zero element $v$ in
$\bbA^{2n}$.  By  Witt's extension theorem (\cite[Theorem
3.9]{Art}), this orbit is $\bbA^{2n} \setminus \{ 0 \}$. Thus the
group $\Sympl_{2n}$ acts on $H_K \ssetminus G_K$ (on the right) via
its natural, linear, action on $\bbA^{2n}_{K}$. Restricting this
action to $T_{\rm gen}$, we reduce our problem to showing that
$\bbA^{2n}/T_{\rm gen}$ is $K$-rational.

The $K$-torus   $T_{\rm gen}$ is split by a $W$-Galois extension
  and the character lattice of $T_{\rm gen}$ is the weight lattice $P({\sf C}_n)$
  with its $W$-action (see \S \ref{thegenerictorus}).
Here $W = (\bbZ/2 \bbZ)^n \rtimes S_n$ is the Weyl group of $G =
\Sympl_{2n}$. Over $L$ we can diagonalize the $T_{\rm gen}$-action
on $\bbA^{2n}$ in some $L$-basis $f_1, \dots, f_{2n}$. Let $\chi_1,
\dots, \chi_{2n}$ be the associated characters of $T_{\rm
gen}\times_K L$.  We have $t \cdot f_i = \chi_i(t) f_i$ for every $t
\in T_{\rm gen}(L)$. These characters are permuted by $W$; denote
the associated $W$-permutation  lattice of rank $2n$ by $P_{1}$.
That is, $W$ permutes a set of generators $a_1, \dots, a_{2n}$ of
$P_{1}$; sending $a_i$ to $\chi_i$, we obtain a morphism of
$W$-lattices
\[ \tau \colon P_1 \to P({\sf C}_n) \, . \]
Since $G = \Sympl_{2n}$ acts faithfully on $\bbA^{2n}$, so does
$T_{\rm gen} \subset G_{K }$;
 hence  $\tau$ is surjective and we obtain a sequence
of $W$-lattices
\[ 0  \to P_2 \to P_1 \to P({\sf C}_n) \to 0 \]
and the dual sequence
\[ 1 \to T_{\rm gen} \to T_1 \to T_2 \to 1 \]
of $K$-tori. The torus $T_{1}$ has a dense open orbit in
$\bbA^{2n}$; identifying $T_1$ with this orbit, we obtain the
following birational isomorphisms of $K$-varieties:
\[ \bbA_K^{2n}/T_{\rm gen} \cong
T_1/T_{\rm gen} \cong T_2. \] It thus remains to show that the
$n$-dimensional torus $T_2$ is rational over $K$. Since every torus
of dimension $\leqslant 2$ is rational, we may assume without loss
of generality that $n \geqslant 3$. We have thus reduced
Proposition~\ref{prop2.type-c} to the following lemma.

\begin{lem} \label{lem.elementary}
Let $n \geqslant 3$,  let $W = W({\sf C}_n) = (\bbZ/2 \bbZ)^n
\rtimes S_n$, let $P$ be a permutation $W$-lattice of rank $2n$ and
let
\[ 0 \to M \to P\stackrel{\varphi}{\to}  P({\sf C}_n) \to 0 \]
be an exact sequence of $W$-lattices. Then there exists a
$W$-isomorphism  between this sequence and sequence {\rm
\eqref{C_{n}sequence}}. In particular, $M$ is a permutation lattice.
\end{lem}

Recall that $P({\sf C}_n)$ has a basis $h_1, \dots, h_n$ such that
$c = (c_1, \dots, c_n) \in (\bbZ/2 \bbZ)^n$ acts on $P({\sf C}_n)$
by $h_i \mapsto (-1)^{c_i} h_i$ and $S_n$ permutes $h_1, \dots, h_n$
in the natural way.

\begin{proof}[Proof of Lemma~{\rm\ref{lem.elementary}}]
Denote the normal subgroup $(\bbZ/2 \bbZ)^n$ of $W$ by $A$. We shall
identify the dual group $A^*=\Hom(A,\bbZ/2\bbZ) $ with $(\bbZ/2
\bbZ)^n$ in the usual way. That is, $(b_1, \dots, b_n) \in (\bbZ/2
\bbZ)^n$ will denote the   additive character $\chi \colon A \to
\bbZ/2 \bbZ$ taking $(a_1, \dots, a_n) \in A$ to $b_1 a_1 + \dots +
b_n a_n \in \bbZ/2 \bbZ$.

We now observe that  $P({\sf C}_n) \otimes \bbQ$ decomposes as the
direct sum of $n$ one-dimensional $A$-invariant $\bbQ$-subspaces,
upon which $A$ acts by the characters
\[ \chi_1 = (1, 0, \dots, 0) \, , \ldots  , \chi_n = (0, \dots, 0, 1) \, . \]
Since the exact sequence
\[ 0 \to M \otimes \bbQ \to P \otimes \bbQ \to P({\sf C}_n)
\otimes \bbQ \to 0 \] of $A$-modules over $\bbQ$ splits, all of
these characters will be present in the irreducible decomposition of
$P \otimes \bbQ$ (as an $A$-module over $\bbQ$). Since $P$ is a
permutation $W$-module, the trivial character will also occur in
this decomposition; denote its multiplicity by $d$, where $1
\leqslant d \leqslant n$. The set $\Lambda$ of the remaining $n - d$
characters is permuted by $S_n$. Since the orbit of a character
$(b_1, \dots, b_n) \in A^*$ has ${n \choose i}$ elements, where $i$
is the number of times $1$ occurs among $b_1, \dots, b_n \in \bbZ/ 2
\bbZ$, and ${n \choose i} \geqslant n$ for any $1 \leqslant i
\leqslant n-1$, we conclude that $(b_1, \dots, b_n)$ can be in
$\Lambda$ only if $i = n$, i.e., $b_1 = \dots = b_n = 1$. In
summary, $P \otimes \bbQ$, viewed as an $A$-module over $\bbQ$, is
the direct sum of the following characters:
\begin
{align} \label{e.characters}
&\mbox{$(0, \dots, 0)$, with multiplicity $d$}, \notag\\
&\mbox{$(1, \dots, 1)$, with multiplicity $n-d$}, \\
&\mbox{$\,\chi_1 = (1, 0, \dots, 0), \dots , \chi_n = (0, \dots, 0,
1)$, each with multiplicity $1$,}\notag
\end{align}
where $d \geqslant 1$ is an integer.

In order to gain greater insight into the $A$-action on $P$ and in
particular to determine the exact value of $d$, we shall now compute
the irreducible decomposition of $P \otimes \bbQ$ (as an $A$-module
over $\bbQ$) in a different way. Recall that $P$ is a permutation
$W$-lattice. In particular, we may write
\[ P \cong \bbZ[A/A_1] \oplus \dots \oplus \bbZ[A/A_r] \, , \]
where $\cong$ denotes an isomorphism of $A$-lattices. Now observe
that, as an $A$-module over $\bbQ$, $\bbQ[A/A_i]$ is the sum of all
those characters of $A$ which vanish on $A_i$. Denote the set of all
such characters by $\Lambda_i \subset (\bbZ/2 \bbZ)^n$. Thus the
$A$-module $P \otimes \bbQ$, which we know is the direct sum of the
$2n$ characters listed in~\eqref{e.characters}, can also be written
as the direct sum of the characters in $\Lambda_1, \dots,
\Lambda_r$. Here some of the characters may appear with multiplicity
$\geqslant 2$; note that $|\Lambda_{1}| + \dots + |\Lambda_r| = 2n$
(here $ |\Lambda_i|$ denotes the order of $\Lambda_i$).
  In other words, the $2n$ characters
listed in~\eqref{e.characters} can be partitioned into $r$ subsets
$\Lambda_1, \dots, \Lambda_r$.

Note that each $\Lambda_i$ is clearly a subgroup of $(\bbZ/2
\bbZ)^n$. On the other hand, since $\chi_l + \chi_m$ is not on our
list~\eqref{e.characters} for any $l \ne m$ (here we are using the
assumption that $n \geqslant 3$), we see that no two of the
characters $\chi_1, \dots, \chi_n$ can be contained in the same
$\Lambda_i$. This implies $r \geqslant n$. After possibly relabeling
the subgroups $A_1, \dots, A_r$  and $\Lambda_1, \dots, \Lambda_r$,
we may  assume that $\chi_i \in \Lambda_i$ for $i = 1, \dots, n$.
Since $\Lambda_i$ is a subgroup of $(\bbZ/2 \bbZ)^n$, each
$\Lambda_i$ should also contain the trivial character. This shows
that the irreducible decomposition of $P \otimes \bbQ$ (as an
$A$-module) contains at least $n$ copies of the trivial character
$(0, \dots, 0)$. We conclude that $n-d = 0$ in~\eqref{e.characters},
$r = n$, $\Lambda_i = \{ (0, \dots, 0), \chi_i \}$, and
\[ A_i = \Ker(\chi_i) =
(\bbZ/2 \bbZ) \times \dots \times (\bbZ/2 \bbZ) \times \{ 1 \}
\times (\bbZ/2 \bbZ) \times \dots \times (\bbZ/2 \bbZ), \] where $\{
1 \}$ occurs in the $i$th position.

We now return to the permutation $W$-lattice $P$. Let $e_1, f_1,
\dots, e_n, f_n$ be a $\bbZ$-basis of $P$ permuted by $W$. As we saw
above, the permutation action of $A$ on this basis is isomorphic to
that on $A/A_1 \cup \dots \cup A/A_n$. After suitably relabeling the
basis elements, we may thus assume that $e_i$ and $f_i$ are the two
elements of our basis fixed by $A_i$. Clearly $A$ permutes $\{ e_i,
f_i \}$. On the other hand, since conjugation by $S_n$ naturally
permutes the subgroups $A_1, \dots, A_n$ of $A$, it also naturally
permutes the (unordered) pairs $\{e_i, f_i \}$.

Since $A_i$ fixes $e_i$ and $f_i$, the elements $\varphi(e_i)$ and
$\varphi(f_i)$ lie in $P({\sf C}_n)^{A_i}= \bbZ h_i$. In other
words, $\varphi(e_i) = m_i h_i$ and $\varphi(f_i) = - m_i h_i$ for
some $m_1, \dots, m_n \in \bbZ$. Since $\varphi$ is surjective, $m_i
= \pm 1$ for each $i$. After interchanging $e_i$ and $f_i$ if
necessary, we may assume that $m_1 = \dots = m_n = 1$. Now $S_n$
permutes both $\{e_1, \dots, e_n \}$ and $\{ f_1, \dots, f_n \}$ in
the natural way. Identifying $e_i \in P$ with $\alpha_i$, $f_i$ with
$\beta_i$, we see that the exact sequence
\[ 0 \to M \to P\stackrel{\varphi}{\to}  P({\sf C}_n) \to 0 \]
is $W$-equivariantly isomorphic to the
sequence~\eqref{C_{n}sequence}.
\end{proof}

This completes the proof of Proposition~\ref{prop2.type-c}, hence of
Theorem~\ref{thm1} (a).

\section{Appendix: $G/T$ versus $T^0$}
\label{sect.appendix }

Let $G$ be a semisimple simply connected group defined over  a field
$K$ and let $T \subset G$ be a $K$-torus.  As we mentioned in
Remark~\ref{rem.G/T-vs-T}, there is no obvious connection between
the (stable) $K$-rationality of $G/T$ and that of $T$.

However, using the results on tori recalled in Section
\ref{sect2.2}, Theorem~\ref{prop4.4}(d) may be rephrased in the
following manner: If $G/T$ is stably $K$-rational, then the dual
torus $T^0$ is stably $K$-rational. Similar,
Proposition~\ref{prop4.4bis} can be rephrased as follows: Suppose
$G$ is a special, split $K$-group and $T \subset G$ is a $K$-torus.
If the dual torus $T^0$ is stably $K$-rational, then $G/T$ is stably
$K$-rational.

One might then wonder if, when $G$ is split (but not necessarily
special), the (stable) $K$-rationality of the dual torus $T^0$
implies that of $G/T$. This is an open question; a positive answer
would yield the stable rationality in the missing case $\sf G_{2}$
in our main Theorem~\ref{thm1} or (equivalently, in
Theorem~\ref{thm1.modified}).

One may go even further and ask whether or not $G/T$ and $T^0$ are
always stably $K$-birationally isomorphic (assuming $G$ is split).
The purpose of this appendix is to show that this stronger assertion
is false.

\begin{prop} \label{prop.counterexample}
There exist a $K$-torus $T$ and a split semisimple simply connected
group $G$ with $T \subset G$ such that $G/T$ is not stably
$K$-birationally isomorphic to the dual torus $T^0$.
\end{prop}

\begin{proof}
Let $K$ be a field of characteristic zero and let
 $L/K$ be a finite Galois extension of fields with group $\Gamma$.
The augmentation map $\bbZ[\Gamma] \to \bbZ $ gives rise to the
exact sequence of $\Gamma$-lattices
$$0 \to I_{\Gamma} \to \bbZ[\Gamma] \to \bbZ \to 0.$$
The norm map $1 \to N_{\Gamma}=\sum_{g \in \Gamma}g \in \bbZ[\Gamma]
$ gives rise to the exact sequence of $\Gamma$-lattices
 $$0 \to \bbZ \to \bbZ[\Gamma] \to J_{\Gamma} \to 0$$
 which is dual to the previous sequence.

Let $T/K$ be the torus with character group $T^*=I_{\Gamma}$. The
character group of $T^0$ is then $(T^0)^* =J_{\Gamma}$.

The  torus $T$ is the $K$-torus $R^1_{L/K}{\bbG}_{m}$  of norm 1
elements in $L$. For $d=[L:K]$  this torus is a maximal torus in
$G={\SL}_{d}$.

 The unramified Brauer group of
the $K$-torus $T^0$  (modulo $\Br K$) is
$\Sha^2_{\omega}(\Gamma,(T^0)^*)$; see~\cite{CTSb}. We have
$$\Sha^2_{\omega}(\Gamma, (T^0)^*)= \Sha^3(\Gamma, \bbZ) =
H^3(\Gamma,\bbZ)$$ (recall that the cohomology of a cyclic group has
period 2 and that $H^1(H, \bbZ)=0$ for any finite group $H$.)

The unramified Brauer group of $G/T$ (modulo $\Br K$)  is
 $\Sha^1_{\omega}(\Gamma, T^*)$; see~\cite{col-k}.
 We have
$$\Sha^1_{\omega}(\Gamma, T^*)= \Ker \Bigl[\hat{H}^0(\Gamma,\bbZ) \to
\prod_{g \in \Gamma} \hat{H}^0(g,\bbZ)\Bigr].$$ This group is
$$\Ker  \Bigl[\bbZ/n_{\Gamma}  \to \prod_{g \in \Gamma} \bbZ/n_{g}\Bigr]$$
where the projection is the natural map, $n_{\Gamma}$ is the order
of $\Gamma$ and $n_{g}$ the exponent of $g \in \Gamma$.

Let us now take $\Gamma=(\bbZ/p)^r$, $r \geqslant 2$. The K\"unneth
formula shows that $ H^3(\Gamma,\bbZ)$ is killed by $p$;
see~\cite[p.\,247]{spanier}. Thus  the group
$\Sha^2_{\omega}(\Gamma, (T^0)^*)$  is killed by $p$.

The group $$\Ker  \Bigl[\bbZ/n_{\Gamma}  \to \prod_{g \in \Gamma}
\bbZ/n_{g}\Bigr]$$ is $$ p \bbZ/p^{r} \bbZ \subset \bbZ/p^r\bbZ \, .
$$ Thus for any $r \geqslant 3$, the group $\Sha^1_{\omega}(\Gamma,
T^*)$ is not killed by $p$.

We conclude that the unramified Brauer group of $G/T$ (modulo $\Br
K$) is not isomorphic to the unramified Brauer group of $T^0$
(modulo $\Br K$). Therefore $G/T$ and $T^0$ are not stably
$K$-birationally isomorphic, as claimed.

\end{proof}
%%%%%%%%%%%%%%%%%%%%%%%%%%%

\section*{Acknowledgements}

The authors are grateful to the referee for contributing
Proposition~\ref{prop.versal0} and for other helpful suggestions.

Colliot-Th\'el\`ene is a researcher at  C.N.R.S. (France).
Kunyavski\u\i \ was supported in part by the Minerva Foundation
through the Emmy Noether Research Institute of Mathematics. Popov
was supported in part by Russian grants {\rus RFFI} 08--01--00095,
{\rus N{SH}}--1987.2008.1, a (granting) program of the Mathematics
Branch of the Russian Academy of Sciences, by ETH Z\"urich, and TU
M\"unchen. He also thanks the Department of Mathematics of UBC for
its hospitality. Reichstein was supported in part by an NSERC
Discovery grant and Accelerator supplement. Colliot-Th\'el\`ene,
Kunyavski\u\i \ and Reichstein thank the organizers of the
Oberwolfach meeting ``Quadratic Forms and Linear Algebraic Groups''
(June 2006). Colliot-Th\'el\`ene and Reichstein also acknowledge the
support of Emory University.

\providecommand{\bysame}{\leavevmode\hbox
to3em{\hrulefill}\thinspace}

\end{document}